\documentclass[4apaper]{amsart} 
\usepackage[cp1251]{inputenc}
\usepackage[english]{babel}

\usepackage{amsmath,amsfonts,amssymb,mathrsfs,amscd,comment,latexsym,bbold}
\usepackage[matrix,arrow,curve]{xy}

\usepackage[usenames]{color}
\usepackage{colortbl}
\usepackage[dvipsnames]{xcolor}

\usepackage{enumitem}

\usepackage{csquotes}

\usepackage[top = 3cm, bottom = 4cm, left = 3cm, right = 3cm]{geometry}

\usepackage{amsthm}

\theoremstyle{plain}
\newtheorem{theorem}{Theorem}
\newtheorem*{nonum-theorem}{Theorem}
\newtheorem{proposition}{Proposition}
\newtheorem{lemma}{Lemma}
\newtheorem{lemma-remark}{Lemma-remark}

\newtheorem{corollary}{Corollary}
\newtheorem{nonum-corollary}{Corollary}

\newtheorem{context}{Context}

\newtheoremstyle{handleNumber}{}{}{\itshape}{}{}{}{\newline}{{\bf #1} \thmnote{#3}}
\theoremstyle{handleNumber}
\newtheorem*{handnum-theorem}{Theorem}

\theoremstyle{definition}
\newtheorem{definition}{Definition}

\newtheorem*{nonum-definition}{Definition}

\theoremstyle{remark}
\newtheorem{remark}{Remark}

\renewcommand{\leq}{\leqslant}
\renewcommand{\geq}{\geqslant}

\newcommand{\A}{\mathbb A}
\newcommand{\PP}{\mathbb P}

\newcommand{\cO}{\mathcal O}
\newcommand{\cI}{\mathcal I}
\newcommand{\GlS}[2]{\Gamma( \PP^{#1},\cO(#2) ) }
\newcommand{\GlSX}[3]{\Gamma( \PP^{#1}_{#2},\cO(#3) ) }

\newcommand{\Fraff}{Fr^{qaf}}

\newcommand{\Fragc}{Fr^{agc}}

\newcommand{\cFraff}{\mathcal Fr^{qaf}}
\newcommand{\cFrsaf}{\mathcal Fr^{saf}}
\newcommand{\wT}[1]{\wedge T^{#1}}

\newcommand{\cFs}{\mathcal F}
\newcommand{\cEs}{\mathcal E}

\newcommand{\SH}{\mathbf{SH}}
\newcommand{\SHd}{\mathbf{SH_\bullet}}
\newcommand{\Hd}{\mathbf{H_\bullet}}

\newcommand{\Spec}{\operatorname{Spec}}
\newcommand{\Ker}{\operatorname{Ker}}
\newcommand{\Coker}{\operatorname{Coker}}

\newcommand{\Image}{\operatorname{Image}}
\newcommand{\der}{d}

\newcommand{\rank}{\operatorname{rank}}
\newcommand{\Supp}{\operatorname{Supp}}

\newcommand{\bfs}{\mathbf s}
\newcommand{\bfe}{\mathbf e}
\newcommand{\bfu}{\mathbf u}
\newcommand{\bfw}{\mathbf w}
\newcommand{\bfom}{\mathbf \ome }
\newcommand{\bfb}{\mathbf b}
\newcommand{\bfc}{\mathbf c}

\newcommand{\Sch}{\mathrm{Sch}}
\newcommand{\Sm}{\mathrm{Sm}}

\newcommand{\Sh}{\mathrm{Sh}}

\newcommand{\Gm}{\mathbb G_m}

\newcommand{\bfa}{\mathbf{a}}\newcommand{\ome}{v}
\newcommand{\oveb}{{\overline{b}}}
\newcommand{\ovew}{{\overline{w}}}
\newcommand{\ovec}{{\overline{c}}}
\newcommand{\bfoveb}{{\overline{\bfb}}}
\newcommand{\bfovec}{{\overline{\bfc}}}
\newcommand{\bfovew}{{\overline{\bfw}}}
\newcommand{\bff}{\mathbf{f}}
\newcommand{\bfh}{\mathbf{h}}
\newcommand{\bfz}{\mathbf{z}}

\newcommand{\ovovs}{\overline{\overline{s}}}

\newcommand{\ovovbfs}{\overline{\overline{\bfs}}}
\newcommand{\ovovbfe}{\overline{\overline{\bfe}}}
\newcommand{\ovovbfu}{\overline{\overline{\bfu}}}

\begin{document}

\title{Geometric models for fibrant resolutions of motivic suspension spectra}
\author{Andrei~Druzhinin}
\thanks{The research is supported by \enquote{Native towns}, a social investment program of PJSC \enquote{Gazprom Neft}.}
\label{abs}
\begin{abstract}
We construct geometric models for the
$\PP^1$-spectrum $M_{\PP^1}(Y)$,
which computes in Garkusha-Panin's theory of framed motives \cite{GP14} 
a motivically fibrant $\Omega_{\PP^1}$ replacement of $\Sigma_{\PP^1}^\infty Y$ in positive degrees
for a smooth scheme $Y\in \Sm_k$ over a perfect field $k$. 
Namely, we get the $T$-spectrum in the category of pairs of smooth ind-schemes that defines $\PP^1$-spectrum of pointed sheaves termwise motivically equivalent to $M_{\PP^1}(Y)$.
\end{abstract}

\maketitle

\section{Introduction}

\subsection{The models given by T-spectra of smooth ind-pairs.}

Consider the category $\Sm^\textrm{pair}_k$ with objects being 
the pairs $(X,U)$ with $X\in Sm_k$ and an open subscheme $U\subset X$ over a filed $k$, and morphisms being morphisms of pairs.
Let $\mathrm{ind}\text{-}\Sm^\textrm{pair}$ be the category
of 
sequences 
\begin{equation}
(X_1,U_1)\to (X_2,U_2)\to \dots (X_i,U_i)\to \dots \end{equation} of closed embeddings of pairs.
We call such sequences 
by smooth ind-pairs.

Denote by $T$ the pair $(\A^1,\A^1-0)$, and 
for $(X,U)\in \Sm^\textrm{pair}_k$ denote the pair 
$(X,U)\wedge T=(X\times \A^1, U\times\A^1\cup X\times(\A^1-0)$. 
The last definition extends in a natural way to ind-pairs as well.
Then we can consider $T$-spectra of ind-pairs, 
by which we mean the sequences
\begin{equation*}\label{eq:SpecTind-pair}
(R^0,\dots R^l,\dots ), f_i\colon R^{l}\wedge T\to R^{l+1},
\end{equation*}
where the terms $R^l$ are ind-pairs and morphisms $f_i$ are morphisms of ind-pairs.
Denote the category of such sequences  
by $\Spec_T \mathrm{ind}\text{-}\Sm^\textrm{pair}$.

Any smooth pair $(X,U)$ defines the Nisnevich sheaf $X/U$ that is a factor sheaf of the sheaves represented by $X$ and $U$.
Then any ind-pair defined a Nisnevich sheaf, and consequently any $T$-spectrum of ind-pairs defines a $T$-spectrum of Nisnevich sheaves.
Thus sine any Nisnevich sheave can be considered as a motivic space we get the functor  
\begin{equation}\label{eq:SpecTpairSH}
\Spec_T \mathrm{ind}\text{-}\Sm^\textrm{pair} \to \SH(k)
\end{equation}
where by $\SH(k)$ we mean the model for the stable motivic homotopy category given by $T$-spectra of motivic spaces.
\begin{theorem}\label{th:wOmeIndPairModelofSigma}
Let $Y\in \Sm_k$ over a perfect filed $k$. 
Then there are 
a $T$-spectrum $M^\prime_{T}(Y)$ in the category $\mathrm{ind}\text{-}\Sm^\textrm{o-pair}_k$
$$
M^\prime_{T}(Y)=(R^0,\dots R^{l}\dots),
$$
such that  $M^\prime_{T}(Y)\simeq \Sigma^\infty_{T}Y$ in $\SH(k)$, where $M^\prime_{T}(Y)$ is considered as object in $\SH(k)$
via the functor \eqref{eq:SpecTpairSH}, and
$$
R^l(Y) \simeq \mathcal Hom_{\Hd(k)}(T, R^{l+1}(Y)).
$$
The construction is natural on the class of $Y$ with an affine neighbourhood for any finite set of points.
\end{theorem}
In particular this implies the representability of stable motivic homotopy groups as unstable ones,
$\pi_{un}^{p,q}(R^{l}(Y))=[\Gm^{\wedge {p-l}}\wedge S^{q-l}, Y]_{\SH(k)}$, $q>0$, though the claim requires us to represent not only the terms $L^l$, but also to represent the structure morphisms $L^l\wedge T\to L^{l+1}$.

Our result is the application of the theory of framed motives \cite{GP14}, which gives in particular the computations of positively motivically fibrant replacements of infinite suspension spectra. 
The theory implies in addition 
that the spectra
$C^*(M^{\prime}_{T}(Y))_f$, where 
$C^*\colon \mathcal F\mapsto \mathcal F(-\times\Delta^\bullet)$, $\Delta$ is the standard simplical object in $\Sm_k$ given by affine spaces,
and 
$(-)_f$ denotes termwise application of the 
Nisnevich local replacement on simplicial pointed sheaves,
gives the $\Omega$-replacement in positive degrees of $\Sigma^\infty_{T} Y$.
We show that the simplicial pointed sheaf $C^*(L^l)$ 
is represented in the category of simplicial schemes, and we expect that more accurate analyse in our technique will show that these schemes are smooth. The representability of such type we have also for the resolution of
the bispectrum $\Sigma^\infty_{\Gm}\Sigma^\infty_{S^1} Y$.


\subsection{Framed motives} 
As noted above our result is the application of the theory of framed motives.

Studying of framed correspondences, and spectra of (pre)sheaves with framed transfers
ware suggested in the unpublished notes \cite{Voev-FrCor} by Voevodsky 
as an alternative approach to the stable motivic homotopy theory 
\cite{V98}, \cite{Jar00}, \cite{MV99}, \cite{M04}
that would be suitable for computational results. 
The idea had grow to the theory of framed motives introduced and developed by Garkusha and Panin \cite{GP14, GP_SHfr} (based on \cite{GP15} and in co-authorship Ananievski \cite{AGP16} and Neshitov \cite{GNP16}). 

In particular, for a smooth scheme $Y$ over a perfect filed $k$  \cite[theorem 4.1]{GP14} gives 
a computation of positively motivically fibrant replacement of 
the infinite $\PP^1$-suspension $\Sigma_{\PP}^\infty Y$.
Namely, it is given by the stably motivically equivalence of the $\PP^1$-spectrum of pointed Nisnevich sheaves
$$
\Sigma_{\PP}^\infty Y\simeq M_{\PP}(Y)_f = (C^*Fr(Y)_f, C^*Fr(Y\wT{1})_f,\dots C^*Fr(Y\wT{l})_f\dots ),
$$
where
$M_{\PP}(-)_f$ is motivically fibrant $\Omega_{\PP^1}$-spectra in positive degrees;
$(-)_f$ is denotes the Nisnevich local fibrant replacement of simplicial (pointed) sheaves;
$C^*\colon \mathcal Y\mapsto \mathcal Hom_{Sh
_\bullet}(\Delta^\bullet, \mathcal Y)$, 
and $Fr$ are the sheaves of framed correspondences.

Let us briefly recall that for $X,Y\in Sm_k$ the elements in $Fr(Y\wT{l})(X)$ are given by the equivalence classes of the data
$c=(Z\hookrightarrow \A^n_X, v\colon V\to \A^n_X,\alpha\colon V\to \A^n\times\A^l\times Y), \text{where}$
\begin{itemize}
\item{}
$V$ is an etale neighbourghood of a closed subscheme $Z$ in $\A^n_X$, and $Z=\A^n_X\times_{(\phi,psi),\A^n\times\A^l} 0$, 
\item{} 
the equivalence relation annihilates 
the choice of $V$, and 
identifies $(Z, V,\alpha)$ with $(0\times Z,\A^1\times V,t_{1},\alpha\circ pr)$, 
where $t_i$ denotes coordinate functions on $\A^n_X$ and $pr\colon \A^1\times V\to V$. 
\end{itemize}

{\it So our question precisely is 
cloud the spectrum $M_{\PP^1}(Y)$ for a smooth scheme $Y$ be represented up to termwise motivic equivalences 
by a spectra of smooth schemes, or pairs of smooth schemes?} 

Since we ask the question up to motivic equivalences
we just need to represent the morphisms $Fr(Y\wT{l})\to Fr(Y\wT{l+1})$ in the category of smooth schemes for a smooth $Y$.
\begin{remark}
Let us note that all mentioned constructions of models does not depend of the properties of the base scheme $S$ (though they depend on the properties of $Y$) at least for an affine $Y$,
while the computations given by the theory of framed motives holds for an an arbitrary smooth schemes but requires the assumption of a perfect base.
\end{remark}

\noindent(\textbf{Representability for $M_{fr}$.})
Firstly, we recall the 
summery of representability results for the case 
from \cite{EHKSY-infloopsp}, 
where the theory of framed motives is studied form the $\infty$-categorical view point.
\begin{nonum-theorem}[\cite{EHKSY-infloopsp}, Theorem 5.1.8]
For a smooth $Y\in \Sm_k$ 
over a perfect filed $k$ such that any finite set in $Y$ has an affine neighbourhood 
there is a pointed smooth ind-scheme $H^{fr}(Y)$ such that
$C^*(H^{fr}(Y))$ is equipped with a canonical structure of $\mathcal E_\infty$-space
such that there is a canonical equivalence 
$\Omega^\infty_T\Sigma^\infty_T Y\simeq (C^*(h^{nfr}(Y))_f)^{gp}$.
In particular 
$\Omega^\infty_T\Sigma^\infty_T Y\simeq (C^*(\mathrm{Hilb}^{fr}(\A^\infty))_f)^{gp}$,
where $\mathrm{Hilb^{fr}}(\A^\infty)$ parametrises finite subschemes $Z$ in $\A^\infty$ with a trivialisation of a (co)normal sheaf $I(Z)/I^2(Z)$. 
\end{nonum-theorem}
The mentioned above result can be considered as the representability for the $S^1$-spectra $M_{fr}(Y)$ called as a framed motive of $Y$, that gives the computation of a positively motivically fibrant replacement of $\Omega^\infty_{S^1}\Sigma^\infty_{G^1}\Sigma^\infty_{S^1}(Y)$ given by the framed motive $M_{fr}(Y)$ \cite[def. 5.2, th. 11.7]{GP14}.

Let us note that the results of the theory of framed motives, namely the mentioned computation \cite[th. 11.7]{GP14} recovered in \cite[cor. 5.5.15]{EHKSY-infloopsp} and additivity theorem \cite[th. 6.4]{GP14}, \cite[proposition 2.2.11]{EHKSY-infloopsp},
yields that
any model for $Fr(Y)$ natural with respect to morphisms 
$$Fr(A \amalg B)\to Fr(A), A =Y\amalg\dots\amalg Y, B =Y\amalg\dots\amalg Y$$
gives such a representability for $\Omega^\infty_T\Sigma^\infty_T Y$.
So up to the mentioned results the theorem is a corollary of the following 
\begin{nonum-theorem}[\cite{EHKSY-infloopsp}, theorem 5.1.5(iii) in combination with corollary 2.2.21]
The pointed sheaves $Fr(Y)$, 
where $Y\in \Sm_S$ is such that any finite set of points has an affine neighbourhood, 
are motivically equivalent in a natural way to the sheaves represented by pointed smooth ind-schemes $H^{fr}(Y)$.
\end{nonum-theorem}
Precisely, it is proven in \cite[theorem 5.1.5(iii)]{EHKSY-infloopsp}
the representability of the (pre)sheaves $Fr^\text{nr}(Y)$ 
that are motivically equivalent to $Fr(Y)$ by \cite[corollary 2.2.21]{EHKSY-infloopsp}.
The (pre)sheaves $Fr^\text{nr}(Y)$ are defined by replacing the etale neighbourghood $V$ in the definition of $Fr(Y)$ by the smallest possible domains for morphisms $(\phi_1\dots \phi_n)$ and $g$.
Namely the the functions $(\phi_1\dots \phi_n)$ are defined on the first order thickening of the support $Z$ in $\A^n_X$ and
$g$ is defined on $Z$. $Fr^\text{nr}(Y)$ are so-called normally framed correspondences firstly shared in the specialists community by A.~Neshitov and they was independently and deeply studied in \cite{EHKSY-infloopsp}.

\vspace{5pt}
In the present text we recover the mentioned above results of \cite{EHKSY-infloopsp} obtaining a model for $Fr(Y)$ with additional properties
\begin{proposition}
The pointed (pre)sheaves $Fr(Y)$,
where $Y\in \Sm_S$ is such that any finite set of points has an affine neighbourhood, 
can be represented up to a motivic equivalences by pointed smooth ind-schemes
$\cFs(Y)$ in a such way that
$\cFs(pt) = \varinjlim_n\cFs_n$ with $\cFs_n \subset \A^N_n$ being a smooth full intersection of a quadrics
such that the projection $\cFs_n\to \A^{\dim\cFs_n}$ is etale.
\end{proposition}
Our result is obtained using another replacement of $Fr(Y)$ that replaces 
the domain of the functions $(\phi_1,\dots \phi_n)$, in distinct to $Fr^\text{nr}(Y)$, by the maximal possible one (namely $\A^n_X$).

\vspace{5pt}
\noindent({\bf Representability for $M^{\Gm}_{fr}$})
In view of the representability question for $M^{\Gm}_{fr}(Y)_f$,
which computes over a perfect filed the motivically fibrant $\Omega$-bi-spectrum replacement of $\Sigma^{\infty}_{\Gm}\Sigma^\infty_{S^1} Y$ \cite[theorem 11.1]{GP14}, models for the sheaves $Fr(Y)$ gives the following
\begin{theorem}
For a smooth scheme $Y$ over a base $S$
such that any finite set of points of $Y$ has an affine neighbourhood
there is a functor $$L\colon \Gamma^{op}\to \Spec_{\Gm} \mathrm{ind}\text{-}\Sm^\textrm{cl-pair}$$ such that
$C^*(L)$ is a 
simplicial object in the category of $\Gm$-spectra of Segal's  $\Gamma$-spaces in the category $\mathrm{ind}\text{-}\Sch^\textrm{cl-pair}$, 
and the corresponding $\Gm$-$S^1$-bi-spectrum is termwise motivically equivalent to $M^{\Gm}_{fr}(Y)$ in positive degrees with respect to $S^1$ direction.
\end{theorem}
Let us note again that this can be immediately deduced using any natural model for $Fr(Y)$, 
the only one point that requires extra checking is the representability in the category of ind-schemes for $Fr(-\times \Delta^n, Y)$, but if we don't care about smoothness then it is not complicated.

One can note that seeking about the representability for the $S^1$-spectra and bi-spectra we actually mean the representability of a $\mathcal E_\infty$-spaces ($\Gamma$-spaces) and moreover even just about the functors $\Gamma^{op}\to \mathrm{ind}\text{-}\Sm$ (or $\mathrm{ind}\text{-}\Sm^\textrm{cl-pair}$). 
Then the structure morphisms of 
the corresponding $S^1$-spectra (bi-spectra) are given by morphisms of simplicial smooth schemes (or we can replace them by non-smooth ind-schemes),
but the author don't know the model for $S^n$ is smooth schemes.
Actually, the representability if $\mathcal E_\infty$-spaces looks being much more natural question.
Nevertheless 
if we want to deduce the data to the pure algebra-geometic objects, then
it looks being much more natural to go to the $T$ (or $\PP^1$) spectra.

\noindent(\textbf{Representability for $M_{\PP^1}$.})
The main result of the article is the theorem 
\begin{theorem}

Let $Y\in \Sm_S$ over a base scheme $S$ of a finite Krull dimension, and assume that any finite set of points in $Y$ has an affine neighbourhood. 
Then $M_{\PP^1}$ is termwise motivically equivalent to a $\PP^1$-spectrum of pointed sheaves defined by a $T$-spectrum of inductive systems of open pairs over $S$. 

Precisely, 
there is a morphism of spectra of simplicial pointed sheaves 
$f\colon M^\prime_{\PP^1}\to M_{\PP^1}$,
$$
M^\prime_{\PP^1} = 
(L^0(Y), L^1(Y),\dots,L^l(Y) )\in {\SHd(S)},\;
L^l(Y)\simeq\varinjlim_{n} (\cFs_n^{l}(Y)  / \cEs_n^l(Y) ),
$$
such that
\begin{itemize}
\item[(-)]  $f$ is a term-wise $\A^1$-Nis-equivalence;
\item[(-)] $\cFs_n^{l}(Y) ) / \cEs_n^l(Y)$
denotes the factor sheaves for the inductive system of 
a pair of smooth $S$-scheme $\cFs_n^{l}(Y)$ and an open subscheme $\cEs_n^{l}(Y)$;
\item[(-)] the structure morphisms $L^{l}(Y)\wedge (\PP^1,\infty) \to L^{l+1}(Y)$ are given by the morphisms of schemes
$e_l\colon \cFs_n^{l}(Y) \times \A^1 \to \cFs_{n-1}^{l+1}(Y) $ 
such that 
$e_l^{-1}(\cEs_{n-1}^{l+1}(Y))\supset (\cEs_n^l(Y)\times \A^1) \cup (\cFs_n^l(Y)\times (\A^1-0))$
in composition with the standard morphism of pointed sheaves $(\PP^1,\infty)\to T$.
\end{itemize}
\end{theorem}
\begin{remark} Note that if we don't care to get the natural model with respect to $Y$ and
if $S=\Spec k$ for a regular noetherian ring $k$, then the case of an arbitrary smooth $Y$ can be reduced to the case of a smooth affine $Y$
by Jouanolou-Tomason's trick. 
\end{remark}
We give also modifications for the case of open or closed pairs $Y$ and $U\subset Y$.
We give two proofs for the theorem. The first one is presented in a sketching way and this proof gives result with more generality, while 
the second proof is more elementary and precise.

%

\subsection{Notations and conventions.}
For a closed subscheme $Z\subset X$ denote by $I(Z)\subset\cO(X)$ the sheaf of ideals of functions vanishing on $Z$. 
For a sheaf of ideals $I\subset \cO(X)$ denote by $Z(I)$ the (nonreduced) vanishing locus of $I$.

For a coherent sheaf $M$ on a scheme $X$ over $S$ we denote by $\Gamma(X,M)$ the coherent sheaf on $X$ that is direct image of $M$ under the canonical projection $X\to S$.

We call the motivic equivalences on the categories of pointed (simplicial) presheaves and sheaves also as $\A^1$-Nis-equivalences. 

\subsection{Acknowledgement}
The author thanks Grigory Garkusha for the consultations with reading of \cite{GP14} and for the consultations on possible criteria of unstable motivic equivalences.
The author thanks Marc Hoyois for the consultation with the representability of the Weil restriction functors.

\section{Framed correspondences and positive $\Omega_{\PP^1}$ spectra.} 

\subsection{Framed correspondences} 

Here is the first our list 
of variations of framed correspondences. 

Let $Y\in \Sm_S$ and $U\subset Y$ is open.
\begin{definition}[(Nisnevich) {framed corr.} $Fr=Fr^\text{Nis}$, \cite{Voev-FrCor}, or def. 2.1 in \cite{GP14}, or def. 2.1.2 in \cite{EHKSY-infloopsp}]\label{def:frcor}
$Fr_n(Y/U\wT{l})$ is a pointed sheaf of sets with the sections
$Fr_n(X,Y/U\wT{l})$ for $X\in \Sch_S$
given by the equivalence classes of the data
$(Z,V, \alpha)$,
where $V\to \A^n_X$ is an etale neighbourghood of a closed subscheme $Z\subset \A^n_X$ finite over $X$, and 
$\alpha=\colon V\to \A^n\times\A^l\times Y$
is a morphism of schemes such that 
$Z = V\times_{\alpha,\A^{n+l}\times Y,(0\times i)} (0\times(Y\setminus U))$, $i\colon Y\setminus U\hookrightarrow Y$;  
all elements $(Z,V, \alpha)$ with $Z=\emptyset$ are pointed;
the equivalence is up to a choice of the etale neighbourhood $V$.
\end{definition}
\begin{remark}
The remarkable Voevodsky's lemma, see \cite[prop 3.5]{GP14}, \cite[cor. A.1.7]{EHKSY-infloopsp} states that $$Fr_n(Y/U\wT{l})=Sh_{Nis,\bullet}( \PP^n_X/\PP^{n-1}_X , (\A^{n+l}/\A^{n+l}-0)\wedge (Y/U) ).$$
\end{remark}
\begin{definition}[first order framed corr. $Fr^\text{1th}$]\label{def:firstordfrcor}
$Fr^{1th}_n(Y/U\wT{l})$ is a pointed sheaf of sets with the sections
on $X\in \Sch_S$ given by the data
$(Z, \alpha)$, where 
$Z\subset \A^n_X$ is closed, and
$\alpha\colon Z_2\to \A^{n+l}\times Y$, where $Z_2=Z(I^2(Z))$, is a morphism of schemes such that 
$Z=Z_2\times_{\A^{n+l}\times Y}  (0\times(Y\setminus U))$. 
The element with $Z=\emptyset$ is pointed.
\end{definition}
\begin{definition}[normally framed corr. $Fr^\text{nr}$, \cite{}, for the case $U=\emptyset$ and $l=0$ 
it is agreed with \cite{AN-HhMWtr} def 4.1 and \cite{EHKSY-infloopsp}, def. 2.2.2.]\label{def:normfrcor}
$Fr^{nr}_n(Y/U\wT{l})$ is a pointed sheaf with the sections 
$Fr^{nr}(X,Y/U\wT{l})$ for $X\in \Sch_S$ given by the data
$(Z,W, \tau,\beta)$, where 
$Z\subset W\subset Z(I(Z)^2)\subset \A^n_X$ are closed,
$\tau\colon I(W)/(I^2(Z))\simeq\cO^n(W)$,
$\beta\colon W\to \A^l\times Y$ such that $Z = W\times_{\A^l\times Y} (0\times (Y\setminus U))$. 
\end{definition}

\begin{definition}[Zariski framed corr. $Fr^{Zar}$.]\label{def:zarfrcor}
For $Y\in Sm_S$, and an open $U\subset Y$,
$Fr^{Zar}_n(X,Y/U\wT{l})$ is a sheaf with the sections
given by the data
$(Z,V, \phi, \beta)$,
where $V\to \A^n_X$ is a Zariski neighbourghood of a closed subscheme $Z\subset \A^n_X$, 
$W\subset V$ is closed,
and 
$\phi\colon V\to \A^n$, $\beta\colon W\to \A^l\times Y$ 
are morphisms of schemes such that 
$W = V \times_{\phi,\A^n,0} 0$, 
$Z = W\times_{\A^l\times Y} 0\times (Y\setminus U)$.
\end{definition}
\begin{remark}
The Zariski framed corr. does not satisfy Zariski version of the Voevodsky's lemma since $g$ is a map $W\to Y$ but not $V\to Y$.
\end{remark}
\begin{definition}[polynomial framed correspondences]
The sections of the presheaf $Fr^{pol}_n(X, Y/U\wT{l})$
are given by the data 
$(\phi,\beta)$ where 
$\phi\colon \A^n_X\to \A^n$,
$\beta\colon W\to \A^l\times Y$,
$W=\A^n_X\times_{\phi,\A^n,0}0$
are such that 
$Z=W\times_{g,\A^l\times Y,0\times i}(0\times (Y\setminus U))$,
is finite over $X$, where
$i\colon Y\setminus U\to Y$ is the canonical embedding.
\end{definition}

\begin{remark}\label{rem:Uemphl0}
If $U=\emptyset$ and $l=0$ the $W$ in the definitions above is an excessive data, and $W=Z$.
In particular, 
$Fr^\text{nr}(X,Y)$ consists of sets $(Z,\tau,g)$ where $Z\subset \A^n_X$ is closed and finite over $X$, $\tau\colon \mathcal O(Z)^n\simeq I(Z)/I^2(Z)$, $g\colon Z\to Y$.

Let us note that $Fr^{nr}_n(Y\wT{l})$ alternatively can be defined by the following (see \cite[def 4.1]{AN-HhMWtr}): 
a pointed sheaf with the sections 
$Fr^{nr}(X,Y/U\wT{l})$ for $X\in \Sch_S$ is given by the data
$(Z, \phi,\psi, g)$, where 
$Z\subset \A^n_X$ is a closed finite over $X$,
$(\phi,\psi)\colon Z(I^2(Z))\to \A^{n+l}$,
$g \colon Z\to Y$ such that $Z = Z(I^2(Z))\times_{\A^{n+l}} 0$. 
\end{remark}
\begin{remark}\label{rem:smallerW}
Under the above definitions $Fr^*(T^{\wedge 1})=Fr^*(\A^1/\Gm)$. 
In the same time we can replace $W$ in the above definitions by the smaller subscheme.

For example
if we define the sections of $Fr^\text{pol}_n(X,Y\wT{l})$
by the data
$(\phi,\psi, g)$,
where 
$(\phi,\psi)\colon \A^n_X\to \A^{n+l}$, and $g\colon W\to Y$, $W=\A^n_X\times_{\A^{n+l}}0$ are such that 
$Z=W\times Y (Y\setminus U)$ is finite over $X$.
Then new $Fr^{pol}(Y/U\wT{l})$ differs form the above one, but it is $\A^1$-Nis-equivalent.
\end{remark}

Denote by $Fr_n^*(l)$ (and $Fr_n^*$) the bi-functor on the product of $Sch^{op}_S$ with the category of pairs $(Y,U)$, which is the full subcategory in the category of arrows of $\Sm_S$, (or on $Sch^{op}_S\times \Sm_S$) given by $Fr^*(X,Y/U\wT{l})$ (or $Fr^*(X,Y)$). 
Then there is a sequence of natural morphisms
\begin{equation}\label{eq:idZarHensnr}
\begin{array}{ccccccccc}
Fr^{pol}_n(l)&\to&
Fr^{Zar}_n(l) &\to &
Fr_n(l) &\to &
Fr^{1th}_n(l) &\to&
Fr^{nr}_n(l) 
\\
(\phi,\beta)&\mapsto&
(Z,W,V,\phi,\beta) &\mapsto& 
(Z,V,(\phi,\beta)) &\mapsto& 
(Z,(\phi,\beta)\big|_{Z_2}) &\mapsto&
(Z,W,\tau,\beta\big|_W)
\\
&&
W \stackrel{\Delta}{=} V \stackrel{\Delta}{=} \A^n_X&&
&&
Z_2=Z(I^2(Z))&&
W\stackrel{\Delta}{=}Z(\phi), \\
&&&&&&&&
\tau \stackrel{\Delta}{=} d\alpha^*
\end{array}\end{equation}

\begin{definition}\label{def:sigmastabilization}
Define the maps $\sigma^*\colon Fr^*_n(X,Y/U\wT{l})\to Fr^*_{n+1}(X,Y/U\wT{l})$
$$\begin{array}{rlcl}
Fr^\text{Nis} &(Z,V,\alpha)&\mapsto &(0\times Z,\A^1\times V,t,\alpha)\\
Fr^\text{1th} & (Z,\alpha)&\mapsto &(0\times Z,t,\alpha)\\
Fr^\text{nr} & (Z,W,\tau,\beta)&\mapsto &(0\times Z,0\times W,(dt,\tau),\beta)\\
Fr^\text{Zar}&
(Z,W,V,\phi, \psi,g)&\mapsto &(0\times Z,0\times W, \A^1\times V,t,\phi,\psi,g)\\
Fr^\text{pol}_n&(\phi,g)&\mapsto &(t,\phi,g)\\
\end{array}$$
Define
$Fr^*(Y\wT{l})=\varinjlim_n Fr_n^*(Y\wT{l})$.
\end{definition}

\begin{remark}
In the list above the bi-functors $Fr^{1th}_*$, and $Fr_*$ define the graded categories,
but others does not.
Bi-functors $Fr^{1th}$ and $Fr$ define 'categories' with the associativity up to a canonical $\A^1$-homotopy,
while all others define 'categories' with a 'composition' up to $\A^1$-homotopy.

Let us note also that
$Fr^{1th}_*(Y)$ is represented by a scheme in a similar way as $Fr^{nr}_*(Y)$ in \cite[theorem 5.1.5]{EHKSY-infloopsp}.
\end{remark}

It is proven in \cite[corollary 2.2.21]{EHKSY-infloopsp} that the morphism $Fr(Y)\to Fr^{nr}(Y)$ is an $\A^1$-Nis-equivalence, and a close statement for the connected components for the correspondences in $T\wT{l}$ is written in \cite[cor. 4.9]{AN-HhMWtr}.
We generalize this by the following.

\begin{proposition}\label{prop:Freq}
For an affine $Y\in \Sm_S$ and open $U\subset Y$ 
the morphisms of the sequence \eqref{eq:idZarHensnr} induces $\A^1$-equivalences on affines
after the $\sigma$-stabilization.

For an arbitrary smooth $Y$ the morphism induces motivic equivalences of sheaves after the $\sigma$-stabilization.
\end{proposition}
\begin{proof}
The claim follows from lemmas \ref{lm:Freq1} and \ref{lm:Freq} in the Appendix C \ref{sect:FrEquiv}.  The second one follows form from lemma \ref{lm:Checkcov}. 
\end{proof}

\subsection{Positively motivically fibrant $\Omega_{\PP^1}$ replacements}

\newcommand{\Smop}{\mathrm{Sm}^{\mathrm{o-pair}} }

Let $k$ be a perfect filed.
\begin{theorem}[Garkusha-Panin, theorem 4.1 in \cite{GP14}]
\label{th:posomegaP1motfibrepl}
Let $Y\in Sm_k$. 
Then 
$\Sigma_{\PP}^\infty\simeq M_{\PP}(-)_f \colon Sm_k\to \SH(k)$,
where
$$M_{\PP}(Y)_f = (C^*Fr(Y)_f, C^*Fr(Y\wT{1})_f,\dots C^*Fr(Y\wT{l})_f\dots )$$
and
$M_{\PP}(-)_f$ lands 
in motivically fibrant $\Omega_{\PP^1}$-spectra in positive degrees, where
$(-)_f$ is
the Nisnevich local fibrant replacement functor $(-)_f$ on the category of simplicial (pointed) sheaves $SSh_{\bullet}$,
and $C^*\colon$ is
the endo-functor $C^*\colon \mathfrak Y\mapsto \mathcal Hom_{Pre_\bullet}(\Delta^\bullet, \mathfrak Y)$ 
on $SSh_{\bullet}$.
\end{theorem}

\begin{corollary}\label{cor:M_P}
Let $Fr^*$ be a bi-functor on $\Sm_S\times \Sm_S^{pair}$ 
that restriction on the category of affine schemes is $\A^1$-equivalent to $Fr(Y\wT{l})$
Then $M_{\PP}^*(Y)_f = (C^*Fr^*(Y)_f, C^*Fr^*(Y\wT{1})_f,\dots C^*Fr^*(Y\wT{l})_f\dots )$
satisfies the same properties as $M_{\PP^1}(Y)$ in theorem \ref{th:posomegaP1motfibrepl}.
\end{corollary}
\begin{proof}
It is enough to consider the case of $\A^1$-equivalences on affines
$Fr^*(Y\wT{l})\to Fr(Y\wT{l})$ or $Fr(Y\wT{l})\to Fr^*(Y\wT{l})$.
Then the morphism $M^*_{\PP^1}(Y)\to M_{\PP^1}(Y) $ (or $M_{\PP^1}(Y)\to M^*_{\PP^1}(Y)$) is a (term-wise) motivic equivalence. Hence 
$\Sigma^\infty_{\PP^1} Y\to M^*_{\PP^1}(Y)$ is a stable motivic equivalence, and  
$M^*_{\PP^1}(Y)$ (and $M^*_{\PP^1}(Y)_f$) is a positively $\Omega_{\PP^1}$-spectra of motivic spaces.

So to get the claim we need to check that $M^*_{\PP^1}(Y)$ is positively motivically fibrant.
By assumption the morphism $C^*Fr^*(Y\wT{l})\to C^*Fr(Y\wT{l})$ (or $C^*Fr(Y\wT{l})\to C^*Fr(Y\wT{l})$) is a section wise (simplicial homotopy) equivalence on affines for smooth affine $Y$. So it is Nis-equivalence for a smooth $Y$, and hence
$C^*Fr^*(Y\wT{l})_f\to C^*Fr(Y\wT{l})_f$ (or $C^*Fr(Y\wT{l})_f\to C^*Fr(Y\wT{l})_f$)
is (sectionwise simplicial homotopy) equivalence. Thus $C^*Fr^*(Y\wT{l})_f$ is positively motivically fibrant.
\end{proof}

\section{Smooth model for $Fr(Y)$ and $M_{\PP^1}$ (the first approach).}

In the section we present our first approach to the construction of the geometric models for the $\PP^1$-spectra $M_{\PP^1}(Y)$ (and $M_{\PP^1}(Y/U\wT{l})$).
The idea is to replace the presheaf $Fr(Y/U\wT{l})$ by the factor $Fr(Y\times\A^l)/Fr(U\times\A^l\cup Y\times{\A^l-0})$ or $Fr(Y\times\PP^l)/Fr(U\times\PP^l\cup Y\times{\PP^{l-1}})$ and 
use an appropriate model for the presheaves $Fr(Y)$, $Y\in \Sm_S$.

The advantages of our model for $Fr(Y)$ with respect to the model obtained in \cite{EHKSY-infloopsp}
are that our model for $Fr(pt)$ is equipped with the canonical (globally defined) etale map to an affine space, and
the structure morphism 
in $\mathbf H(S)$ of presheaves
$Fr(Y)\wedge (\PP^1,\infty)\to Fr(Y\wT{1})$
is representable by morphisms of schemes.
The representability of the Weil restriction functor is used in the present model for $Fr(Y)$ for of the same reason as in \cite{EHKSY-infloopsp}, namely to parametrize the regular maps $g\colon Z\to Y$ in the definition of framed correspondence.

\subsection{Standard idempotent framed corr.}

The replacement $Fr^{pol}$ extends the framing functions $\phi_i$, $i=1\dots n$, defining the map $V\to \A^n$ in the definition of framed correspondences to the maximal possible neighbourhood of the support $Z$, namely $\phi_i\in \cO(\A^n_X)$. 
But it leads to that the vanishing locus $W= Z(\phi_1,\dots \phi_n)$ could not be finite over $X$ itself in general, and even for the case of pairs $(Y,\emptyset)$, $l=0$, the vanishing locus $Z(s_1,\dots s_n)$ would intersect infinity $\PP^{n-1}_X$, for any $s_i\in \Gamma(\PP^n_X,\cO(d_i))$, $\phi_i=x_i/x_\infty^{d_i}$.

Nevertheless there is a modification of the definition such that $W$ is finite over $X$, and moreover all $W_i=Z(\phi_{i+1},\dots \phi_n)$ are finite over $\A^{i}_X$ under the projection $\A^n_X\to \A^i_X$ with respect to first $i$ coordinates, and furthermore $\phi_i$ are polynomials with leading terms defining empty vanishing locus on $\PP^{n-1}_X$. This property is useful with repsect to the representability question, since it guarantee that $Z$ is finite over $X$ by pure algebraically condition (and even linearly algebraically).

Since in this section we are concentrated on the case of pairs $(Y,\emptyset)$, $l=0$, 
the only one definition we actually need for the rest part of the section is def. \ref{def:stidfrcor};
in the definition \ref{def:idfrcor} we just write how to apply this approach for the general case.

\begin{definition}[standart idempotent framed corr. $Fr^\text{st-id}$]\label{def:stidfrcor}
\item[(st. id. corr.)]
Define $Fr^\text{st-id}_{n}(Y)$ as the sheaf 
with sections on $X\in \Sch_S$ given by the data 
$(Z,s_1,\dots s_n,g)$, 
$s_i\in \Gamma(\PP^n_X,\cO(3^{n-i}))$, $s_i\big|_{\PP^{n-1}_X}=t_i^{3^{n-i}}$, 
such that $Z(s_1\dots s_n)=Z\amalg \hat Z$ for some $\hat Z$, 
and $g\colon Z\to Y$.

\item[(hyperbolic corr.)]
Denote by $Fr^\text{st-hyp}_{n}(X,Y)\subset Fr^\text{st-id}_{n}(X,Y)$ consisting of such $(Z,s_1,\dots s_n,g)$ that $Z=Z(s_1,\dots s_n)$. 

\item[(stabilisation)]
Define the maps
\begin{gather*}
s_i^\prime =s_i(t_\infty^{2d_i}-s_i^2), d_i=3^{n-i}, \text{ for }1\leq i<n,
\\
\begin{array}{llllclc}
Fr^\text{st-hyp}_{n}(X,Y)&\to& Fr^\text{st-hyp}_{n+1}(X,Y)&\colon&
(s_1,\dots s_n)&\mapsto& (s_1^\prime,\dots s_n^\prime, t_{n+1})\\
Fr^\text{st-id}_{n}(X,Y)&\to& Fr^\text{st-id}_{n+1}(X,Y)&\colon&
(Z,s_1,\dots s_n,g)&\mapsto& (Z\times 0, s_1^\prime,\dots s_n^\prime, t_{n+1}, g)\\
\end{array}
\end{gather*}
Denote $Fr^\text{st-hyp}(X)=\varinjlim_{n}Fr^\text{st-hyp}_n(X)$,
$Fr^\text{st-id}(X,Y)=\varinjlim_{n}Fr^\text{st-id}_n(X,Y)$.
\end{definition}

To explain the place of this definition with more details let us give few more replacements of fr. corr.
The dashed arrows in the diagram exists only in the unstable level, 
and for the case of $U=\emptyset$ all arrow are  $\A^1$-Nis equivalences on the unstable level, but in general only un-dashed are so.
$$\xymatrix{
Fr^\text{st:id}\ar[r]\ar[dd]& Fr^\text{id}\ar@{-->}[dr]\ar[dd]\ar[drr] \\
&& 
Fr\ar[r] & Fr^\text{nr}\\
Fr^\text{st:e}\ar[r]& Fr^\text{e}\ar@{-->}[ur]\ar[urr] \\
}$$

\begin{definition}[idempotent framed corr. $Fr^\text{id}$]\label{def:idfrcor}
$\phantom{p}$\\
Define $Fr^\text{st-id}_{n}(Y)$ as the sheaf 
with sections on $X\in \Sch_S$ given by the data 
$(Z,\phi_1,\dots \phi_n,g)$, 
$Z\subset \A^n_X$ is finite, $\phi_i\in \cO(\A^n_X)$, $g\colon Z\to Y$,
such that $Z(\phi_1\dots \phi_n)=Z\amalg \hat Z$ for some $\hat Z$.
\\
Define $Fr^\text{id}_{n}(Y/U\wT{l})$ as the sheaf 
with sections on $X\in \Sch_S$ given by the data 
$(Z,\phi_1,\dots \phi_n,g)$, 
$Z\subset \A^n_X$ is finite, $\phi_i\in \cO(\A^n_X)$, $g\colon Z\to Y$,
such that $Z(\phi_1\dots \phi_n)=Z\amalg \hat Z$ for some $\hat Z$.
\end{definition}

\begin{definition}[(standart) equational framed corr. $Fr^{eq}$]\label{def:stidfrcor}
$\phantom{p}$\\
Define $Fr^\text{eq}_{n}(Y)$ as the pointed sheaf 
with sections on $X\in \Sch_S$ given by the data 
$(e,\phi_1,\dots \phi_n,g)$, 
$\phi_i\in \cO(\A^n_X)$,
$(e^2-e)\big|_{Z(\phi_1,\dots \phi_n)}=0$,
$g\colon Z(e,\phi_1,\dots \phi_n)\to Y$;
\\
Define $Fr^\text{eq}_{n}(Y\wT{l})$ as the pointed sheaf 
with sections on $X\in \Sch_S$ given by the data 
$(e_1,e_2,\phi_1,\dots \phi_n,\beta)$, 
$\phi_i\in \cO(\A^n_X)$,
$(e_1^2-e_1)\big|_{Z(\phi_1,\dots \phi_n)}=0$,
$g\colon Z(e,\phi_1,\dots \phi_n)\to Y$;
$(e_2^2-e_2)\big|_{\beta^{-1}(0\times(Y\setminus U))}=0$;
\\
Define $Fr^\text{st-eq}_{n}(Y)$ as the pointed sheaf 
$s_i\in \Gamma(\PP^n_X,\cO(3^{n-i}))$, $s_i\big|_{\PP^{n-1}_X}=t_i^{3^{n-i}}$, 
and $e\in \cO(\A^n_X)$
such that $(e^2-e)\big|_{Z(s_1\dots s_n)}=0$,
and $g\colon Z(e,s_1,\dots s_n)\to Y$.\item[(standard equational corr.)]
Define $Fr^\text{st-eq}_{n}(Y)$ as the pointed sheaf 
with sections on $X\in \Sch_S$ given by the data 
$(e,s_1,\dots s_n,g)$, 
$s_i\in \Gamma(\PP^n_X,\cO(3^{n-i}))$, $s_i\big|_{\PP^{n-1}_X}=t_i^{3^{n-i}}$, 
and $e\in \cO(\A^n_X)$
such that $(e^2-e)\big|_{Z(s_1\dots s_n)}=0$,
and $g\colon Z(e,s_1,\dots s_n)\to Y$.

\item[(stabilisation)]
Define the maps
\begin{gather*}
s_i^\prime =s_i(t_\infty^{2d_i}-s_i^2), d_i=3^{n-i}, \text{ for }1\leq i<n,
\\
\begin{array}{llllclc}
Fr^\text{st-eq}_{n}(X,Y)&\to& Fr^\text{st-eq}_{n+1}(X,Y)&\colon&
(e,s_1,\dots s_n,g)&\mapsto& (e, s_1^\prime,\dots s_n^\prime, t_{n+1}, g)\\
\end{array}
\end{gather*}
Denote $Fr^\text{st-hyp}(X)=\varinjlim_{n}Fr^\text{st-hyp}_n(X)$,
$Fr^\text{st-id}(X,Y)=\varinjlim_{n}Fr^\text{st-id}_n(X,Y)$.
\end{definition}

\begin{proposition}\label{lm:Freq1}
For an affine $Y\in \Sm_S$ and open $U\subset Y$ 
there are natural $\A^1$-equivalences of sheaves
$Fr^\text{st-id}(Y)\to Fr(Y)$, and
$Fr^\text{eq}(Y/U\wT{l})\to Fr^\text{id}(Y/U\wT{l})\to Fr(Y/U\wT{l})$.

For any $Y\in \Sm_S$ and open $U\subset Y$  there is natural $\A^1$-Nis-equivalence of motivic spaces 
$Fr^\text{st-id}(Y)\to Fr(Y)$, 
and equivalences 
$Fr^\text{eq}(Y/U\wT{l})\to Fr^\text{id}(Y/U\wT{l})\to Fr(Y/U\wT{l})$.
\end{proposition}
\begin{proof}
Clearly the sheaf $Fr^\text{st-id}(Y)$ satisfy the closed glueing.
By lemma \ref{lm:Freq} and proposition \ref{prop:LiftCriteria}, to prove the first claim of the lemma
it is enough to prove the lifting property with respect to closed embeddings of affines for the morphism $Fr^\text{st-id}\to Fr^\text{g-nr}$.

Let $c=(Z,W,\tau,g)\in Fr^\text{g-nr}_n(X,Y)$, $(Z_0,s_1^0,\dots s_n^0,g^0)\in Fr^\text{st-id}_n(X_0,Y)$ define the element in 
$Fr^\text{g-nr}_n(X,Y)\times Fr^\text{g-nr}_n(X_0,Y) Fr^\text{st-id}_n(X_0,Y)$.

Applying lemma \ref{lm:EqforNeigh} 
we can get for all large enough $b$ 
an etale neighbourhood 
$r_{i}\in \Gamma(\PP^{n+m},\cO(d_{n+i}))$, $d_{n+i}=3^{r+b-i}$, $i=1\dots m$,
$r_{i}\big|_{\PP^{n+m}}=t_{n+i}^{d_{n+i}}$,
$C\subset Z(f_1,\dots f_m)\subset \A^{n+m+b}_X$,
where $f_i\in cO(\A^{n+m+b}$ is the inverse image of $r_i/t_\infty^{d_{n+i}}$ under the projection,
such that
$Z(f_1\,dots f_m)-C\to \A^n_X$ is an etale neighbourhood of $W\amalg_{X_0\times_X Z}Z_0$.

Then for all large enough $b$ there are sections $s_i\in \Gamma(\PP^{n+m+b},\cO(d_i))$, $d_{n+i}=3^{n-i}$,
$s_{i}\big|_{\PP^{n+r+b}}=t_{i}^{d_i}$, $s_i/t_\infty^{d_i}\big|_{Z_0}=s_i^0$, $s_i/t_\infty^{d_i}\big|_{Z(I^2(W))}$ is agreed with $\tau$.
Now 
$(s_1,\dots s_n,s_{n+1},\dots s_{n+m},t_{n+m+1}(1- t_{n+m+1}^{3^{b-1}}-1),\dots t_{n+m+b-1}(1-t_{n+m+b-1}^2),t_{n+m+b},g)\in Fr^\text{st-id}_{n+m+b}(X,Y)$ is the required lift of $\sigma^{m+b}c$.
,where $s_{n+i}$ are sections suc that $s_{n+i}/t_{\infty}^{d_{n+i}}$ is equal to the inverse image of $f_{i}$. 

In a similar way we get the equivalence 
$Fr^\text{id}(Y/U\wT{l})\to Fr^\text{g-nr}(Y/U\wT{l})$.
The equivalence $Fr^\text{eq}(Y/U\wT{l})\to Fr^\text{id}(Y/U\wT{l})$ in view of proposition \ref{prop:LiftCriteria} follows immediate from Chinese remainder theorem. 
\end{proof}

\subsection{ind-smooth model for $Fr_n(Y)$}

According to the above definition we get the sequence of forgetful functors
$$Fr^\text{st:id}(Y)\to Fr^\text{st:id}(pt)\to Fr^\text{st:hyp}(pt),$$ 
where the first one is just the composition with the canonical morphism $Y\to pt$ and the second one forgets the choice of the disjoint component.

\begin{proposition}\label{prop:st:id:Model}
For any $Y\in \Sm_S$, such that any finite set of points in $Y$ has an affine neighbourhood,
the sheaves $Fr^\text{st:id}_{n}(Y)$, %
and $Fr^\text{st:hyp}_{n}(Y)$ are represented in $\Sm_S$ by smooth schemes
$\mathcal Fr^\text{st:id}_{n}(Y)$, 
$\mathcal Fr^\text{st:hyp}_{n}(pt)$, 
such that 
there is a sequence of morphisms 
\begin{equation}\label{eq:Freseqsmet}
\mathcal Fr^\text{st:id}_{n}(Y)\to\mathcal Fr^\text{st:id}_{n}(pt)\to\mathcal Fr^\text{st:hyp}_{n}(pt)\simeq \A^{N_{n}}
\end{equation}
with the first morphism being smooth and the second one being etale.

The natural isomorphisms 
$Fr^\text{st:id}(Y)=\varinjlim_n \mathcal Fr^\text{st:id}_{n}(Y)$, and $Fr^\text{st:hyp}(pt)=\varinjlim_n \mathcal Fr^\text{st:hyp}_{n}(pt)$ are compatible with the morphisms in sequence \eqref{eq:Freseqsmet}.
\end{proposition}
\begin{proof}
Denote $N_n=\sum_{i=1}^{n}(\dim \Gamma(\PP^n,\cO(3^{n-i}))-\dim \Gamma(\PP^{n-1},\cO(3^{n-i})))$.
Then since by definition \ref{def:stidfrcor} 
the elements of $Fr^\text{st:hyp}_{n}(X,pt)$ are the sets of sections 
$(s_1,\dots s_n)$, 
$s_i\in \Gamma(\PP^n_X,\cO(3^{n-i}))$, $s_i\big|_{\PP^{n-1}_X}=t_i^{3^{n-i}}$,
there is a one-to-one correspondences between $Fr^\text{st:hyp}_{n}(X,pt)$ and the $X$-points of $\A^{N_{n}}$.

Next, if we are given with the element in $Fr^\text{st:hyp}_{n}(X,pt)$ all what we need to define an element
in $Fr^\text{st:id}_{n}(X,pt)$ is to choose a disjoint component of the schemes $Z(s_1,\dots s_n)$. 
Since $Z(s_1,\dots s_n)$ is finite and flat over $X$ and $\mathcal O(Z(s_1,\dots s_n))\simeq \mathcal O(X)^{d}$ for some $d$ 
the second arrow is represented by an etale morphism by lemma \ref{lm:etid}.

The first arrow in \ref{eq:Freseqsmet} is the Weil restriction functor and it is represented by smooth morphism by \cite[section 7.6, proposition 5(h)]{}. 
\end{proof}

\begin{lemma}\label{lm:etid}
Let $f\colon Z\to X$ be a finite locally free morphism. Denote by $V\to X$ the total space of the vector bundle over $X$ defined by the coherent sheaf $A=f_*(\cO(Z))$. Denote by $Id(Z)\to X$ the closed subscheme in $V$ defined by the equation $e^2-e=0$, where $e^2$ is the square with respect to the algebra structure on the sheaf $A$.
Then the morphism $Id(Z)\to X$ is etale.
\end{lemma}
\begin{proof}

Firstly let's note that $Id(Z)$ is quasi-finite over $X$, since 
there is only finite set of idempotents in a finite dimensional algebra over a filed. 
Then it follows that 
the endo-morphism $w\colon V\to V\colon e\mapsto e^2-e$ is quasi-finite over some neighbourhood $W$ of the zero section $0_X\subset V$. Since $V$ is smooth over $X$ it follows that $w$ is flat over $W$. 
Hence $Id(Z)$ is flat over $X$.

Since the algebra $A$ is commutative the differential $d(e^2-e)$ is equal to $2e-1$. 
Since the vanish locus $Z(e^2-e,2e-1)\subset Z(e^2,2e-1)$ is empty, $\Omega_{Id(Z)/X}=0$. Now since idempotents satisfy the non-separable descent, 
it follows that the morphism $V\to X$ is unramified. 

(Comment of the non-separable descent:
Since the disjoint components of the reduced subscheme are the same as reduced subschemes of the disjoint components, it follows that idempotents in the finite dimensional algebra over any filed satisfy the descent with respect to local artin algebras. Since the number of disjoint components of the product is equal to the product of numbers of the disjoint components, it follows that the descent with respect to local artin algebras implies the descent with respect to non-separable extensions.)
\end{proof}

\subsection{Construction of the Ind-smooth models for $Fr(Y\wT{l})$ and $M_{\PP^1}(Y)$.}

\newcommand{\Shvd}{{{Shv}_{\bullet}}}

Here we give the the general construction of the smooth model for $M_{\PP^1}(Y/U)$ and a sketch of the proof.
\begin{lemma}\label{lm:ConeAPrepl}
Assume that the base scheme $S$ is of a finite Krull dimension.
There are motivic equivalences of pointed Nisnevich sheaves
$$\begin{array}{llll}
Fr(Y\wedge T^l)= &Fr(Y\times(\A^l/(\A^l-0) )\simeq &
Fr(Y\times(\A^l//(\A^l-0) ),
\\
Fr(Y\wedge T^l)= &Fr(Y\wedge(\PP^l/(\PP^l-0) )\simeq&
Fr(Y\times(\PP^l//(\PP^l-0)) )\simeq&
Fr(Y\times(\PP^l//\PP^{l-1} )),
\\
Fr(Y\wedge T^l)= &Fr(Y\times(\PP^1/(\PP^1-0))^\wedge{l} )\simeq&
Fr(Y\times(\PP^1//\infty)^\wedge{l} ).
\end{array}
$$
Moreover these morphisms being restricted to affines became $A^1$-equivalences.
\end{lemma}
\begin{proof}
The first isomorphism follows by the precise computation. 

The second equivalence 
$Fr(Y\times(\A^l//(\A^l-0) )\simeq Fr(Y\times(\A^l/(\A^l-0) )$
follows from the Nisnevich equivalence 
$Fr^{1st}(Y\times(\A^l//(\A^l-0) )\simeq Fr^{1st}(Y\times(\A^l/(\A^l-0) )$, which is straightforward. And similarly we get the second equivalences in other rows.

The last isomorphism in the second row follows since $Fr(Y\times (\PP^l-0) ) \to
Fr(Y\times\PP^{l-1} )$ is an $\A^1$-equivalence (by cirteria \ref{lm:HomEq}).
\end{proof}

\begin{remark}
In the above proof we do not use the equivalence $M_{fr}(Y/U)\simeq M_{\PP^1}(Y//U)$ in general case.

The second equivalences actually follows from the constructions and arguments form \cite{GNP16} and \cite{MotSmAfPair}.
Let us note that results in \cite{GNP16} and \cite{MotSmAfPair} is formulated for the case of sheaves of abelian groups because of the aim of the computation of the framed motive, but the arguments can be translated to the case of sheaves of sets.

Namely the equivalence is given by Nisnevich equivalence $Fr(Y\times(\A^l//(\A^l-0) )\simeq Fr^{qf}(Y\times(\A^l/(\A^l-0) )$,
and $\A^1$ equivalence $Fr^{qf}(Y\times(\A^l/(\A^l-0) )\simeq Fr(Y\times(\A^l/(\A^l-0) )$. 
where $Fr^{qf}(X,Y/U\wT{l})$ are correspondences such that the subscheme $Z(\phi)$ is quasi-finite over $X$, where $\phi\colon V\to \A^n$ is framing functions.  
Actually, the first morphism induces the simplicial equivalences on the section on henselian local schemes. 
The $\A^1$-homotopy inverse for the second morphism is
the morphism of the telescope simplicial set corresponding to the filtrations on $Fr^*(Y\times(\A^l/(\A^l-0) )$ defined similarly as the ones in the proof of \cite[theorem 6]{MotSmAfPair} in view of \cite[proposition 3]{MotSmAfPair}.
\end{remark}

\begin{remark}
\end{remark}

Let $Y\in \Sm_S$ and $U\subset Y$ is an open subscheme.
Define the $T$-spectrum of pairs of schemes
\begin{gather*}
M^\prime_{T,n}(Y/U) = (L^0_n, \dots L^l_n,\dots ),
L^l_{n} = Fr^{\mathrm{st:id}}_{n}(Y\times\A^l)/ Fr^{\mathrm{st:id}}_{n}(U\times\A^l\cup Y\times(\A^l-0)) ,
\\ 
\begin{array}{lcl}
L^l_{n}\wedge T &\to& L^{l+1}_{n} \\ 
((Z;\phi_1,\dots \phi_n;\psi_1,\dots\psi_l;g), x)&\mapsto &
(Z;\phi_1,\dots \phi_n;\psi_1,\dots\psi_l,x;g)
\end{array}
\end{gather*}
where
$\phi_i,\psi_j\colon \A^n_X\to \A^1$, $Z(\phi,\psi)=Z\amalg \hat Z$,
$g\colon Z\to Y$,
$x\in \A^1$.
We consider the spectrum $M^\prime_{T,n}(Y/U)$ as a $\PP^1$-spectrum of a pointed sheaves via the canonical morphism $(\PP^1,\infty)\to T\in \Sh_{\bullet}$.  

Let $Y\in \Sm_S$ and $U\subset Y$ is a closed smooth subscheme.
Define the $\PP^1$-spectrum of pairs of schemes
\begin{gather*}
M^\prime_{\PP^1,n}(Y/U) = (L^0_n, \dots L^l_n,\dots ),
L^l_{n} = Fr^{\mathrm{st:id}}_{n}((Y,U)\wedge(\PP^1,\infty)^{\wedge l}) ,
\\ 
\begin{array}{lcl}
L^l_{n}\wedge (\PP^1,\infty) &\to& L^{l+1}_{n} \\ 
((Z;\phi_1,\dots \phi_n;\psi_1,\dots\psi_l;g), [x_0\colon x_\infty])&\mapsto &
(Z;\phi_1,\dots \phi_n;\psi_1,\dots\psi_l,x;(g,[x_0\colon x_\infty]))
\end{array}
\end{gather*}

Let 
$$
M^\prime_{T}(Y/U)\simeq \varinjlim_n M^\prime_{T,n}(Y/U),\;
M^\prime_{\PP^1}(Y/U)\simeq \varinjlim_n M^\prime_{\PP^1,n}(Y/U)
$$
be the termwise inductive limit of spectra of pair of schemes. We consider them as a spectra of points Nisnevich sheaves that are factor-sheaves represented by pairs.

Now immediate 
The first two claims follow directly from lemma \ref{lm:ConeAPrepl}. 
All the rest follows form corollary \ref{cor:M_P} and \cite[theorem 4.1]{GP14}.

\begin{theorem}
For any $Y\in \Sm_S$ over a finite Krull dimensional scheme $S$ such that any finite set of points in $Y$ has an affine neighbourghood, and $U\subset Y$ be an open subscheme, 
the canonical morphism 
$g(M^\prime_T(Y/U))\to M_{\PP^1}(Y/U)$ is a termwise motivic equivalence, where $g\colon \Spec_T \mathrm{Sh}_{\bullet}\to \Spec_{\PP^1}\mathrm{Sh}_{\bullet}$ is the standard forgetful functor.
If $U\subset Y$ is a smooth closed subscheme then the canonical morphism $M^\prime_{\PP^1}(Y/U)\to M_{\PP^1}(Y)$ is a termwise motivic equivalence.

Assume $S=\Spec k$ for a perfect filed $k$.
Then the canonical morphism 
$Y\to Fr^\text{st:id}(Y)$
induces a motivic equivalence of $\PP^1$-spectra of pointed sheaves
$\Sigma^\infty_{\PP^1} Y/U\to M^\prime_{\PP^1}(Y/U)$.
The $\PP^1$-spectra $C^*( M^\prime_{\PP^1}(Y/U))_f$ is a positively motivically fibrant $\Omega_{\PP^1}$ spectrum.

Moreover if $U=\emptyset$ or $Y$ is quasi-affine then 
$g(M^{\prime}_{T}(Y/U))\to M_{\PP^1}(Y/U)$ are a termwise motivic equivalences, where $M_{\PP^1}(Y/U)$ is defined similarly to $M_{\PP^1}$ (\cite[section 4]{GP14}) using $Fr(Y/U^wT{l})$.
\end{theorem}
\begin{proof}
For the case of $U=\emptyset$ the $g(M^\prime_T(Y))\to M_{\PP^1}(Y)$ and $M^\prime_{\PP^1}(Y/U)\to M_{\PP^1}(Y/U)$ follow directly from lemma \ref{lm:ConeAPrepl}, and all the rest follows form corollary \ref{cor:M_P} and \cite[theorem 4.1]{GP14}. 
The last statement for the case of quasi-projective $Y$ is the consequence of \cite{MotSmAfPair}.

The equivalences $g(M^\prime_T(Y/U))\to M_{\PP^1}(Y/U)$ and $M^\prime_{\PP^1}(Y/U)\to M_{\PP^1}(Y/U)$ follow from the ones for $Y$ and $U$. The properties of positively motivically $\Omega$-spectra $g(M^\prime_T(Y/U))\to M_{\PP^1}(Y/U)$ and $M^\prime_{\PP^1}(Y/U)\to M_{\PP^1}(Y/U)$ follows by the same arguments as in the proof \cite[theorem 4.1]{GP14} (or alternatively they can be deduced from the result of \cite[theorem 10.1]{GP14}).
All properties of $M^\prime_{T}(Y/U)$ and $M^\prime_{T}(Y/U)$ follows by the same arguments as properties of $M_{\PP^1}(Y)$ in \cite{GP14}. Actually some of them are already stated and proven inside the proof of \cite[theorem 4.1]{GP14}.
\end{proof}

Consider the category $\Sm^\textrm{o-pair}_k$ with objects being 
the pairs $(X,U)$ with $X\in Sm_k$ and an open subscheme $U\subset X$ over a filed $k$, and morphisms being morphisms of pairs.
Consider the category $\Sm^\textrm{cl-pair}_k$ with objects being pairs $(X,Z)$ where $Z\subset X$ is a closed subscheme that is the union $Z= Z_1\amalg \dots Z_n$ for smooth closed subschemes $Z_i$ such that for any $I\subset \{1,dots n\}$, $\cap_{i\in I} Z_i$ is smooth.
We can see that the categories $\Sm^\textrm{*-pair}_k$ are equipped with closed symmetry monoidal structure
$$
(X_1,Y_1)\wedge (X_2,Y_2) = (X_1\times X_2, Y_1\times X_2\cup X_1\times Y_2).
$$

Let $\mathrm{ind}\text{-}\Sm^\textrm{*-pair}$ be the category
with objects being sequences 
\begin{equation}
(X_1,U_1)\to (X_2,U_2)\to \dots (X_i,U_i)\to \dots \end{equation} 
of closed embeddings in $\Sm^\textrm{*-pair}_k$
and morphisms being morphisms of sequences.
We call such sequences \eqref{eq:ind-pairs} by ind-pairs, precisely either open smooth ind-pairs, either closed ind-pairs.
The monoidal structure extends in a natural way to ind-pairs as well.

Denote by $T$ the pair $(\A^1,\A^1-0)\in \mathrm{ind}\text{-}\Sm^\textrm{o-pair}$.
Then we can consider $T$-spectra of ind-pairs, 
by which we mean the sequences
\begin{equation*}\label{eq:SpecTind-pair}
(R^0,\dots R^l,\dots ), f_i\colon R^{l}\wedge T\to R^{l+1},
\end{equation*}
where the terms $R^l$ are ind-pairs and morphisms $f_i$ are morphisms of ind-pairs.
Denote the category of such sequences  
by $\Spec_T \mathrm{ind}\text{-}\Sm^\textrm{o-pair}$.

Similarly consider $\PP^1$ as the pair $(\PP^1,\infty)\in \mathrm{ind}\text{-}\Sm^\textrm{cl-pair}$,
and define the category $\Spec_{\PP^1} \mathrm{ind}\text{-}\Sm^\textrm{cl-pair}$ of $\PP^1$-spectra
of closed ind-pairs.

Any smooth pair $(X,U)$ defines the Nisnevich sheaf $X/U$ that is a factor sheaf of the sheaves represented by $X$ and $U$.
Then any ind-pair defined a Nisnevich sheaf, and consequently any $T$-spectrum ($\PP$-spectrum) of ind-pairs defines a $T$-spectrum  ($\PP$-spectrum) of Nisnevich sheaves.
Thus sine any Nisnevich sheave can be considered as a motivic space we get the functor  
\begin{equation}\label{eq:SpecTpairSH}
\Spec_T \mathrm{ind}\text{-}\Sm^\textrm{o-pair} \to \SH(k), \; 
\Spec_{\PP^1} \mathrm{ind}\text{-}\Sm^\textrm{cl-pair} \to \SH(k)
\end{equation}
where by $\SH(k)$ we mean the model for the stable motivic homotopy category given by $T$-spectra ($\PP^1$-spectra) of motivic spaces.

Now using the representability obtained in proposition \ref{prop:st:id:Model} we get 
\begin{theorem}
Let $Y\in \Sm_k$ over a perfect filed $k$. 
Then there are 
a $T$-spectrum $M^\prime_{T}(Y)$ in the category $\mathrm{ind}\text{-}\Sm^\textrm{o-pair}_k$ and 
a $\PP^1$-spectrum $M^\prime_{\PP^1}(Y)$ in $\mathrm{ind}\text{-}\Sm^\textrm{cl-pair}_k$
$$\begin{array}{cc}
M^\prime_{T}(Y)=(R^0,\dots R^{l}\dots), &M^\prime_{\PP^1}(Y)=(L^0,\dots L^{l}\dots)
\end{array}$$
such that 
$M^\prime_{T}(Y)\simeq \Sigma^\infty_{T}Y$ and $M^\prime_{\PP^1}(Y)\simeq \Sigma^\infty_{\PP^1}Y $ 
in 
$\SH(k)$,  and
$$\begin{array}{cccc}
R^l(Y) \simeq \mathcal Hom_{\Hd(k)}(T, R^{l+1}(Y)),& L^l(Y) \simeq \mathcal Hom_{\Hd(k)}( (\PP^1,\infty), L^{l+1}(Y)), & l>0.
\end{array}$$
The construction is natural on the class of $Y$ with an affine neighbourhood for any finite set of points.
\end{theorem}

Finally, let us outline the construction that gives us the model with all terms being a indunctive sequences of pairs $(X,Z)$ $X\in \Sm_k$ and $Z$ is closed smooth subschemes.
But we need to replace the notion of $\PP^1$-spectra by the notion of twisted $\PP^1$-spectra, where the suspension $(X,Z)\wedge \PP^1$ is replaced with the nontrivial $(\PP^1,\infty)$ bundle $(\widetilde X,\widetilde Z)\to (X,Z)$ and a morphism $\overline{ X}\to X\times \A^1$ that induces isomorphism $\overline X\times (X-Z)\simeq \A^1\times(X-Z)$, where $\overline X$ is the corresponding $\A^1$ bundle over $X$.

Alternatively we can say that the twisted spectrum is the sequence 
$$(L^0,L^1,\dots L^l,\dots ),  L^l\leftarrow \widetilde{L}^l\to \widetilde L^{l+1},$$
where 
the morphisms in the right formula are morphisms in $\mathrm{ind}\text{-}\Sm^\textrm{cl-pair}$,
and
$L^l = \varinjlim_n (X^l_n,Z^l_n)$, $Z^l_n$ are smooth,
and 
$L^l\leftarrow \widetilde{L}^l$ is a $(\PP^1,\infty)$ bundle,
with a rational morphism of ind-pairs $\widetilde{L}^l\dasharrow L^l\wedge \PP^1$ that in an isomorphism of schemes $\widetilde{X}^l_n\times_{} (X^l_n-Z^l_n)\simeq (X^l_n\wedge \PP^1)\times (X^l_n-Z^l_n)$.

Then the model for $\Sigma_{\PP^1}(Y/Z)$ in the twisted $\PP^1$-spectra is given by the spectrum with terms
Define the $\PP^1$-spectrum of pairs of schemes
\begin{gather*}
M^{\prime\prime}_{\PP^1,n}(Y/Z) = (L^0_n, \dots L^l_n,\dots ),\\
L^l_{n} = Fr^{\mathrm{st:id}}_{n}((Y,Z)\wedge(\PP^l,\PP^{l-1}) , 
\widetilde{L}^{l+1}_{n} = Fr^{\mathrm{st:id}}_{n}((Y,Z)\wedge(Bl_{(1,1,\dots ,1,0)}(\PP^{l+1}),W^l))
\end{gather*}
where $W^l$ is the proper preimage of $\PP^l\subset\PP^{l+1}$ consisting of points $(0,x_1,\dots x_n)$.
and morphisms given by the corresponding morphism for $\PP^{l+1}$, $\PP^{l}$, and $Bl_{(1,\dots 1,0)}(\PP^{l+1})$.

\section{Quasi-affine model for $Fr(-\times P,Y/U\wT{l})$.}

In the section we assume the following context notations
\begin{context}
\label{context:Y0}
Let $Y$ be an affine scheme over $S$, $N\in \mathbb Z$ is even, and
$y_i\in \GlS{N}{d_y}$, $i=1\dots N-r$, such that
$Z(y_1,\dots y_{N-r})\cap \A^N=Y\amalg \hat Y$.
Let $P$ be an affine scheme over $S$ and $p_i\in \GlS{M}{d_p}$, $i=1,\dots q$,
$Z(p_1,\dots p_q)=V\subset \PP^N$.

For any $n>N$ we consider 
$Y$ as the subscheme $0\times Y\times 0\subset \A^M\times \A^n$,
define $y_j\in \Gamma(\PP^{M+n},\cO(d_y))$, $N-r< j\leq n-r$, by $y_j=t_{j+r+M}(t_\infty-t_{j+r+M})^{d_y-1}$,
and consider
$y_j$, $j\leq N-r$ 
as sections in $\Gamma(\PP^{M+n},\cO(d_y))$ via the rational projection map 
$\PP^{M+n}\dasharrow \PP^{N}\colon 
(t_\infty,t_1,\dots t_n)\mapsto (t_\infty,t_{M+1},\dots t_{M+N})$.
Consider the sections $p_i$ as a sections in $\Gamma(\PP^{M+n},\cO(d_p))$ via the rational projection map $\PP^{M+n}\dasharrow \PP^{M}\colon (t_\infty,t_1,\dots t_n)\mapsto (t_\infty,t_{1},\dots t_{M})$. 

Precisely this means that 
the homogeneous polynomial defining $p_j$ on $\PP^{M+n}_X$ is given by the same formula as the homogeneous polynomial defining $p_j$ on $\PP^M$, and
the homogeneous polynomial defining $y_j$ on $\PP^{M+n}_X$ is given by the same formula as the homogeneous polynomial defining $y_j$ on $\PP^N$ but under the substitution $t_i\dasharrow t_{M+i}$.  

Denote by $\overline Y$ and $\overline{\hat Y} $ the closure of $\hat Y$ and $Y$ in $\PP^N$ (or $\PP^n$).

\end{context}
In addition we will work in the following context
\begin{context}\label{context:YU0oc}
Under the assumptions of the context \ref{context:Y0} let $U\subset Y$ be
an open.
\end{context}

\subsubsection{The sheaves $\Fraff_{n,d}$.}

\begin{definition}\label{def:Fraff}
Choose some even integers 
$d, d_e, d_u, d_b, d_w\in \mathbb Z$.
Under the context \ref{context:YU0oc} 
for any $n\in \mathbb Z$, $n>N$,
define the pair of quasi-affine schemes
$\cFraff(P,Y/U\wT{l}) = (\cFraff_\alpha(-\times P,Y\times\A^{l}),\Fraff_\alpha(-\times P,U\times\A^{l}\cup Y\times(\A^l-0)))$,
$\alpha = (n, d, d_e,  d_u, d_b, d_w)$
by the following.

Consider the $S$-affine scheme $\cFs_\alpha(P,Y\wT{l})$ parametrising the
vectors $a = (e,s, u, c, b, w, 
\ome, f, h, z, \oveb, \ovec,\ovew)$
\begin{gather*}\label{eq:sectcond}
\begin{array}{lllccc}
&&& e&\in& \GlSX{M+n}{X}{d_e},\\
s &=& (s_1,\dots s_n,s_n+1,\dots s_{n+l}), &s_i&\in& \GlSX{M+n}{X}{d_e},\\
u &=& (u_1,\dots u_n), &u_j&\in& \GlSX{M+n}{X}{d_u},\\
b &=& (b_1,\dots b_{n+l}), &b_i&\in& \GlSX{M+n}{X}{d_b},\\
w &=& (w_{i,j})_{i=1,\dots n+l}^{j=1,\dots n-r}, &w_{i,j}&\in& \GlSX{M+n}{X}{d_w},\\
\ome &=& (\ome_1,\dots, \ome_{n-r}), &\ome_j&\in& \GlSX{M+n}{X}{d_\ome },\\
c &=& (c_{i,k})_{i=1,\dots n+l}^{k=1,\dots n}, &c_{i,k}&\in& \GlSX{M+n}{X}{d_c},\\
f &=& (f_1,\dots f_{n-r}), &f_{i}&\in& \GlSX{M+n}{X}{d_f},\\
\oveb &=& (\oveb_1,\dots \oveb_{q}), &b_i&\in& \GlSX{M+n}{X}{d_\oveb},\\
\ovew &=& (\oveb_1,\dots \oveb_{q}), &w_{i,j}&\in& \GlSX{M+n}{X}{d_\ovew},\\
\ovec &=& (c_{i,k})_{i=1,\dots q}^{k=1,\dots n}, &b_i&\in& \GlSX{M+n}{X}{d_\ovec},
\end{array}
\\
d_c=d_b,\, 
d_\ome =d_w+d - 2 d_e,\, 
d_f =d +d_b - d_y^2
\\
d_\oveb = d_\ovec = d_b+d-d_p,\,
d_\ovew = d_w+d-d_p,\,
d_h = d_b+d-2d_u,
\end{gather*}
such that
\begin{equation}\label{eq:defFraff}
\begin{array}{rclclclcl}
t_\infty^{d+d_b-2d_e}(t_\infty^{d_e} - e) e &=& \sum\limits_{i=1}^{n+l} b_i s_i &+& \sum\limits_{i=1}^{q} \overline{b}_{i} p_i &+& \sum\limits_{i=1}^{n-r} f_j y_j^2, &&\\
u_k t_\infty^{d_c+d-d_u}&=&\sum\limits_{i=1}^{n+l} c_{i,k} s_i &+&\sum\limits_{i=1}^{l} \overline{c}_{i,k} p_i &+& \sum\limits_{j=1}^{n-r} z_{j,k} y_j^2 \\
u_k\big|_{\PP^{n-1}_X}&=&t_k^{d_u}\big|_{\PP^{n-1}_X},\\
t_\infty^{d+d_w-d_y} y_j &=& \sum\limits_{i=1}^{n+l} w_{i,j} s_i &+& \sum\limits_{i=1}^{q} \overline{w}_{i,j} p_i &+& \sum\limits_{k=1}^{n} h_{k,j} u_k  &+& \ome _j e^2 ,
\end{array}\end{equation}
where 
$\PP^{n-1} = Z(t_\infty,t_1,\dots t_M)\subset\PP^{M+n}$.

Define $\cFraff_\alpha(X\times P,Y\times\A^{l})$ as an open subscheme of $\cFs_\alpha(P,Y\wT{l})$ 
\begin{multline*}
\cFraff_\alpha(X\times P,Y\times\A^{l}) = 
\{(e,s, u, w, \ome , c)\in \cFs_\alpha |
\Gamma(\PP^{M+n}_X,\cO(d))\twoheadrightarrow \Gamma(Z(I^2(Z(e,s,u))),\cO(d)),\\
Z(s_1,\dots s_{n+l})\cap \overline{\hat Y}=\emptyset, 
Z(s_1,\dots s_{n+l})\cap (\overline{Y}\setminus Y)=\emptyset\}.
\end{multline*}

Define $\cFraff_\alpha(X\times P,U\times\A^{l}\cup Y\times(\A^l-0))$ as an open subscheme of $\cFraff_\alpha(-\times P,Y\times \A^l)$
\begin{multline*}
\cFraff_\alpha(X\times P,U\times\A^{l}\cup Y\times(\A^l-0)) = \\
\{ (e,s, u, w, \ome , c)\in \cFraff_\alpha(-\times P,Y\times \A^l)|\, 
\A_X^{M+n}\times_{Y\times\A^{n+l+1}} (Y\setminus U)\times 0=\emptyset \} \end{multline*}
where the morphism $\A_X^{M+n}\to Y\times\A^{n+l}$ 
is given by 
\begin{equation}\label{eq:Frammaps}
\begin{array}{ccc}
pr_{Y}\colon \A_X^{M+n}&\to& Y\\(x,t_1,\dots t_{M+n})&\mapsto& (t_{M+1},\dots t_{M+N}),
\end{array}\quad \begin{array}{ccc}
\A_X^{M+n}&\to& \A^{n+l+1}\\(x,t_1,\dots t_{n+l})&\mapsto& (s_1/t_\infty^{d}, s_{n+l}/t_\infty^{d}, e/t_\infty^{d_e}).
\end{array}\end{equation}
\end{definition}

\begin{remark}
In the case of the base filed the condition $Z(s_1,\dots s_{n+l})\cap \overline{\hat Y} =\emptyset$ can be replaced by the equation $s_1\big|_{\hat Y}=t_\infty^d$. 

Moreover we can delete the equations $y_i$ for , in other words we can work with the subscheme $Y\times\A^{N-n}$ instead of $Y\times 0$ in $\A^n=\A^N\times\A^{N-n}$.
So this model is defined with less number of equations. 
\end{remark}

\begin{definition}
Denote by 
$\Fraff(-\times P,Y\times\A^l)$ and $\Fraff(-\times P,U\times\A^l\cup Y\times(\A^l-0) )$
the sheaves represented by the corresponding schemes, and denote by 
$\Fraff(-\times P,Y/U\wT{l})$ the Nisnevich factor-sheaf represented by the pair 
$\cFraff(P,Y/U\wT{l})$.
\end{definition}

\begin{lemma}\label{lm:FrafftpFr}
Let $a = (e,s, u, c, b, w,\ome, f, h, z, \oveb, \ovec,\ovew)\in Map(X,\cFs_\alpha(P,Y/U\wT{l})$. 
Denote $Z = Z(s,e)\cap (X\times P \times Y\times \A^{n-N})$, and $Z^{(s)} = (Z(s)\cap \A^{M+n}_X)\times_{\A^M} P$.
Then the following hold
\begin{itemize}
\item[{\rm\{1\}}] $Z$ is finite over $X\times P$ and is contained in $X\times P\times \A^{n-N}\times Y$; 
\item[{\rm\{2\}}] $Z^{(s)} =Z\amalg \hat Z$, $Z\subset Z(e)$, $\hat Z\subset Z(t^{d_e}_\infty -e)$;
\end{itemize}
If $a\in \cFraff(X\times P,Y\wT{l})$ then 
\begin{itemize}
\item[{\rm\{3\}}] 
$Z(\overline{\overline s})\subset X\times P\times \A^n$,
and $Z$ is disjoint component of $Z(\overline{\overline s})$, and $Z=Z(\overline{\overline e},\overline{\overline s},\overline{\overline u})$
\end{itemize}
where
$\overline{\overline s}$ and $\overline s$ denote the inverse images of $s$ with respect to 
morphisms
$X\times P\times \overline Y\to \PP^{M+n}_X$
$X\times P\times Z(I^2(\overline Y))\to \PP^{M+n}_X$,
and similarly for $\overline e$, $\overline{\overline e}$, and 
$\overline u$, $\overline{\overline u}$.
\end{lemma}
\begin{proof}
It follows from the definitions that
\begin{equation}\label{eq:ZZ(s)ov}
Z = Z(\overline e,\overline s)\cap X\times P\times \A^n,
Z^{(s)} = Z(\overline s)\cap X\times P\times \A^n.
\end{equation}

The second and the third equations of \eqref{eq:defFraff} imply that 
\begin{equation}\label{eq:ovovs} 
Z(\overline{\overline s})= Z(\overline{\overline s},\overline{\overline u})\amalg (Z(\overline{\overline e},\overline{\overline s})\cap X\times P\times \PP^{n-1}).
\end{equation} 
Then the first equation of \eqref{eq:defFraff} implies that 
$Z(\overline{\overline s},\overline{\overline u}) = Z(\overline{\overline e},\overline{\overline s},\overline{\overline u})\amalg Z(t_\infty^{d_e}-\overline{\overline e},\overline{\overline s},\overline{\overline u})$.
Finally, the last equation of \eqref{eq:defFraff} implies that
\begin{equation}\label{eq:subsetY}
Z(\overline{\overline e},\overline{\overline s},\overline{\overline u})\subset X\times P\times Y
\end{equation} (we mean $X\times P\times Y \times 0$, see context \ref{context:Y0}).
Since by definition $Z({\overline e},{\overline s},{\overline u})\subset X\times P\times Y$, it follows now that
$Z(\overline{\overline e},\overline{\overline s},\overline{\overline u})=Z({\overline e},{\overline s},{\overline u})\amalg {\hat Z^\prime}$.
Thus combining with \eqref{eq:ovovs} we see that
$Z(\overline{\overline s}) = Z({\overline e},{\overline s},{\overline u})\amalg H$. 

Using \eqref{eq:ovovs} again we see that 
$Z(\overline{\overline s},\overline{\overline u}) = Z(\overline{\overline s})\cap X\times P\times \A^n$, 
and 
$Z({\overline s},{\overline u}) = Z({\overline s})\cap X\times P\times \A^n$.
Using \eqref{eq:ZZ(s)ov} we see 
$Z({\overline s},{\overline u})  = Z^{(s)}$, and
$Z({\overline e},{\overline s},{\overline u}) = Z$.
Hence $Z$ is finite over $X\times P$ and by \eqref{eq:subsetY} $Z\subset X\times P\times Y$.
This is \{1\} and \{2\} follows from \eqref{eq:ovovs}. The point \{3\} follows immediate form the above and the definition of $\cFraff$.
\end{proof}

\begin{lemma}
Let $n$ and $l$ be integers. 
Then for any $d_u\in \mathbb Z$
there is $h\in\mathbb Z$ such that for all $d>h$,
an affine scheme $X$,
sections $u=(u_1,\dots u_n)\in\GlSX{n}{X}{d_u}$, $u_i\big|_{\mathbb P^{n-1}_X}=t_i^{d_u}$,
and a closed subscheme $Z\subset Z(u)$,
the restriction homomorphism $\Gamma(\PP^n_X, \cO(d)^l ) \to \Gamma(Z, \cO(d)^l )\simeq \Gamma(Z, \cO )$ is surjective.
\end{lemma}
\begin{proof}
It follows from the conditions on $u$ that 
$Z(u)$ and consequently $Z$ is finite over an affine scheme $X$.
Hence $\Gamma(Z(u),\cO(d))\simeq\cO(Z(u))\twoheadrightarrow \cO(Z)\simeq \Gamma(Z,\cO(d))$.
So without less generality we can assume $Z=Z(u)$.

Consider the $k$-affine space $\Gamma_u$ with closed points being $\{u=(u_i)|u_i\in \Gamma(\PP^n), u_i\big|_{\PP^{n-1}}=t_i^{d_u}\}$. Let $\bfu\in \Gamma(\PP^n_{\Gamma_u},\cO(d_u)^n)$ denotes the universal vector section.
Then by lemma \ref{cor:SerresTh} for all large enough $d$ the homomorphism $\Gamma(\PP^n_{\Gamma_u}, \cO(d)^l ) \to \Gamma(Z(\bfu), \cO(d)^l )\simeq \Gamma(Z, \cO )$ is surjective.

For any affine $X$ 
denote by $p_X\colon \PP^n_X\to X$ the canonical projection,
and denote by $q_{u,X}\colon Z(u)\to X$ for a given vector of sections $u$.
Then since $X$ is affine 
$\Gamma(\PP^n_X, \cO(d)^l ) \twoheadrightarrow \Gamma(Z(u), \cO(d)^l )$ 
iff
$(p_X)_*(\cO(d)^l ) \twoheadrightarrow (q_{u,X})_*(\cO(d)^l )$.
In the same times we have $$(p_X)_*(\cO(d))=\Upsilon^*((p_{\Gamma_u})_*(\cO(d))), (q_{u,X})_*(\cO(d))=\Upsilon^*((p_{\bfu,\Gamma_u})_*(\cO(d))).$$
Hence 
$(p_{\Gamma_u})_*(\cO(d)^l ) \twoheadrightarrow (q_{\bfu,\Gamma_u})_*(\cO(d)^l )$
implies
$(p_X)_*(\cO(d)^l ) \twoheadrightarrow (q_{u,X})_*(\cO(d)^l )$.
\end{proof}

\begin{lemma}\label{lm:SurFrafftoFr}
Let $X$ and $P$ are affine schemes. Assume $Z$ is a closed subscheme in $\A^n_{X\times P}$ finite over $X\times P$, 
and $s=(s_i)_{i=1,\dots l}$ be a vector of sections $s_i\in \GlSX{M+n}{X}{d}$ such that 
$(Z(s)\cap \A^{M+n}_X)\times_{\A^M} P=Z\amalg \hat Z$,
$\hat Z\cap (\overline{\hat Y}\cup (\overline{Y}\setminus Y))=\emptyset$, 
and 
$Z\subset X\times P\times \A^{n-N}\times Y$.

Then 
$\exists h_e, h_u$ $\forall d_e>h_e, d_u>h_u$ $\exists h_b, h_w$ $\forall d_b>h_b, d_w>h_w\in \mathbb Z$
there is a vector of sections $a = (e,s, u, c, b, w,\ome, f, h, z, \oveb, \ovec,\ovew)$ as in def. \ref{def:Fraff} 
such that equalities \eqref{eq:defFraff} hold. 
Moreover the integers $h_e, h_u, h_b, h_w$ can be chosen independently on the affine schemes $X, P, Y$.
\end{lemma}
\begin{proof}
Denote by $ \overline{Z}$ and $ \overline{Z}$ the closures of $Z$ and $\hat Z$.
Applying Serre's theorem \ref{cor:SerresTh} 
to the closed subschemes $\PP^{n-1+M}_X\amalg \overline{Z}$ and $\overline{Z}\cup \overline{\hat Z}$ in $\PP^n_{X\times P}$ we 
find $h_u$ and $u$, 
$h_e$ and $e$. The choice in given by Serre's theorem is independent form $Y$ and moreover by lemma \ref{lm:StSurIdeals} 
one can see that the choice is independent form the affine schemes $X$ and $P$.
Then using lemmas \ref{lm:StSurIdeals}, \ref{lm:lift}, 
we find $h_b.h_w$ and $(c, b, w,\ome, f, h, z, \oveb, \ovec,\ovew)$.
\end{proof}

\begin{definition}\label{def:hdHpd}
For any $n$ and $d$  let us chose and fix some
$h_e(n,d)$ and $h_u(n,d)$ and setting $d_e=h_d$, $d_u=h_u$ choose some $h_b(n,d)$ and $h_w(n,d)$
such that lemma \ref{lm:SurFrafftoFr} is satisfied,
and the homomorphisms
\begin{gather*}
\begin{array}{rcl}
\Gamma(\PP^{n}_{\tilde X}, \cO(d_e))&\to& \Gamma(Z(\tilde s),\cO(d_b+d))
\\
\Gamma(\PP^n_{\tilde X},\cO(d_b)^{n+l}\oplus \cO(d_f)^{n-r} )&\to& 
\Gamma(\PP^n_{\tilde X},\cI(Z(\ovovs))(d_b+d)) \colon\\ 
(\der b,\der f)&\mapsto& \sum_{i=1}^l \der b_i \tilde s_i + \sum_{j=1}^{n-r} \der f_j y_j^2,
\\
\Gamma(\PP^n_{\tilde X},\cO(d_w)^{n+l}\oplus\cO(d_\ome))&\to& \Gamma(\PP^n_X,\cI(\tilde Z)(d_w+d)) \colon\\ 
(\der w,\der \ome)&\mapsto& \sum_{i=1}^{n+l} \der w_{i,j} s_i + \der\ome_j e
\\
\Gamma(\PP^n_X,\cO(d_u)) &\twoheadrightarrow& \Gamma(Z(\ovovs)\amalg \PP^{n-1}_X,\cO(d_u)) \colon\\ 
(\der u_k)&\mapsto& (u_k\big|_{Z(\ovovs)},u_k\big|_{\PP^{n-1}_X}) 
\end{array}
\end{gather*}
are surjective,
for any 
affine $\tilde X$,
$\tilde Z\subset Z(\tilde s)$, 
$\tilde s\in \Gamma(\PP^n_{X\times P},\cO(d)^{n+l})$,
$\tilde e\in \Gamma(\PP^n_{X\times P},\cO(d_e))$ 
$\tilde u\in \Gamma(\PP^n_{X\times P},\cO(d_e))$ 
$\ovovs = \tilde s\big|_{X\times \overline {Y}}$. 
Such $h_*$ exists by lemmas \ref{lm:StSurIdeals}, \ref{lm:lift}.

Denote by $\cFs_{n,d}\cFs_{\alpha(n,d)}$, $\cFraff_{n,d}=\cFraff_{\alpha(n,d)}$, 
where $\alpha(n,d) = (n,d,d_e(n,d), d_u(n,d), d_b(n,d), d_w(n,d) )$
and similarly defined sheaves $\Fraff_{n,d}$.
Denote by 
$\cFs_{n}=\cFs_{n,3^n}$, $\cFraff_{n}=\cFraff_{n,3^n}$, $\Fraff_{n}=\Fraff_{n,3^n}$.
\end{definition}

\subsubsection{Stabilisation.}

Consider the universal vector of sections 
$\bfa = (\bfe,\bfs, \bfu, \bfc, \bfb, \bfw,\bfom, \bff, \bfh, \bfz, \bfoveb, \bfovec,\bfovew)$
where $\cFs=\cFs_{n,d}(P,Y\wT{l})$ for some $n,d$.
Denote $Z=Z(\bfs,\bfe,\bfu)$.

Consider the sections 
$$
\bfs^\prime_i\in \GlSX{n+1}{\cFs}{3d}, i=1\dots, l+1, \, 
\bfs^\prime_i = \bfs_i (t_\infty^{2d}- \bfs_i^2), \text{ for } i\leq l, \text{ and } \bfs^\prime_{l+1} = t_{n+1} (t_\infty^{3d-1}-t_{n+1}^{3d-1}), 
$$
where $d_e=h^e_{n,d}$, $\check{d}_e=d_e+2dl + 3d-1$, and where $\bfs_i$ and $\bfe$ are considered as sections on $\PP^{n+1}_\cFs$ in a standard way via the inclusion $\PP^n_\cFs\to \PP^{n+1}_\cFs$. 
Then one can see 
$
Z(\bfs^\prime) = Z \coprod Z^\prime$, $\bfs^\prime_i\big|_{Z(I^2(Z))}=\bfs_i t^{2d}_\infty\big|_{Z(I^2(Z))}, 
$
where $Z$ is considered as a closed subscheme in $\PP^{n+1}_\cFs$ via the inclusion $\PP^{n}_\cFs\hookrightarrow\PP^{n+1}_\cFs$.

Then by lemma \ref{lm:SurFrafftoFr}  there is a vector of sections $
\bfa^\prime = (\bfe^\prime,\bfs^\prime, \bfu^\prime, \bfc^\prime, \bfb^\prime, \bfw^\prime,\bfom^\prime, \bff^\prime, \bfh^\prime, \bfz^\prime, \bfoveb^\prime, \bfovec^\prime,\bfovew^\prime)
\in \Fraff_{n+1,3d}(\cFs\times P,Y\wT{l})$
The section $\bfa^\prime$ gives us a regular map $\varphi_{n,d}\colon \cFs_{n,d}(P,Y\wT{l})\to \cFs_{n+1,3d}(P,Y\wT{l})$ that induces the map 
$$ \varphi_{n,d}\colon \cFraff_{n,d}(P,Y\wT{l})\to \cFraff_{n+1,3d}(P,Y\wT{l})$$ 

\begin{definition}\label{def:stabFraff}
For any $Y$ and $P$ as in the context \ref{context:Y0}
Define the pairs of ind-schemes $\cFraff(P,Y\wT{l})$ and the Nisnevich sheaf $\Fraff(P,Y\wT{l})$
$$\cFraff(P,Y\wT{l})= \varinjlim\limits_{n} \cFraff_{n,3^n}(P,Y\wT{l}),\,
\Fraff(P,Y\wT{l})= \varinjlim\limits_{n} \Fraff_{n,3^n}(P,Y\wT{l})$$ 
with the morphisms given by $\varphi_{n, 3^n}$.
Then 
the sheaf $Fr^{aff}(-\times P,Y\wT{l})$ is represented by the pair of ind-schemes $\varinjlim_n\cFraff_{n,3^n}(P,Y\wT{l})$.
\end{definition}

\section{Smoothness}

\begin{context}
\label{context:Y00}
Assume the context \ref{context:Y0} and let $q=0$, $M=0$, $P=pt$.
Moreover assume that $Y$ is affine and $r=\dim Y$.
\end{context}

Denote $\Fraff(Y\wT{l})=\Fraff(pt,Y\wT{l})$,
$\cFraff(Y\wT{l})=\cFraff(pt,Y\wT{l})$.
The goal of the section is the following
\begin{proposition}\label{prop:smoothness}
For any $n,d\in \mathbb Z$, and a smooth affine scheme $Y$ as in context \ref{context:Y00},   
the quasi-affine scheme $\cFrsaf_{n,d}(Y\wT{l})$ is smooth.
\end{proposition}

\subsection{Preliminarily lemmas}

\begin{lemma}\label{lm:SmSurRegRestriction}
Let $f=(f_i)_{i=1,\dots r}$, $y_i\in \cO(\A^n_X)$, for an affine $k$-scheme $X$.
Suppose $Z(f)=Y\amalg \hat Y\subset \A^n_X$, and $Y$ is smooth over $X$.
Let $\varphi=(\varphi_i)_{i=1,\dots l}$, $\varphi_i\in \cO(\A^n_X)$ such that $Z(\varphi)=Z\amalg \hat Z$, $Z \subset Y$, 
and let
\begin{equation}\label{eq:As=y}
\sum\limits_{i=1,\dots l} w_{i,j} \cdot (\varphi_i\big|_{Z}) = f_j\big|_{Z(I^2(Z))},\;
w_{i,j}\in \cO(Z), 1\leq i\leq l, 1\leq j\leq r.\end{equation} 

Then the homomorphism $A\colon \cO(Z)^l\to \cO(Z)^r$
given by the matrix $A=(w_{i,j})_{{ i=1,\dots l } \atop { j=1\dots r }}$ is surjective.
\end{lemma}
\begin{proof}
Nakoyama's lemma implies that the homomorphism $A$ is surjective if and only if 
for each point $x\in Z$ the rank of the matrix $A(x)$ is equal to $r$.
Assume that $z\in Z$ is a point such that $\rank A(z)<r$. Hence there is a linear function of $k(z)^r$ (a raw) $u=(u_i)_{i=1,\dots r}$ such that $u\cdot A(z)=0$.

Consider the sections $\der f_j$, $j=1\dots r$ of the cotangent vector space in $\A^n$ at $z$ 
$\hat T_z=I(z)/I(z)^2$ defined by the functions $f_i$, i.e. $\der f_j = f_j\big|_{Z(I^2(z))}\in I(z)/I(z)^2$, $j=1,\dots r$.
Since $Y$ is smooth, the sections $\der f_j$ are linearly independent. 
Consider then the sections of the fibre of the conormal sheaf of $Z$ at $z$: $\overline{f}_j\in i^*(\hat N_Z)$, $i\colon z\to Z$, $\hat Z=I(Z)/I^2(Z)$. Then $i^*(\hat N_Z)\simeq (I(Z)/I^2(Z))\otimes_{\cO(\A^n_X)} (\cO(\A^n_X)/I(z))$ and the homomorphism $(I(Z)/I^2(Z))\otimes_{\cO(\A^n_X)} (\cO(\A^n_X)/I(z))\to I(Z)/I^2(z)\to I(z)/I^2(z)$ implies that $\overline f_j$ are linearly independent too.

In the same time equalities \ref{eq:As=y} implies that $u\cdot \overline f = u\cdot A(z) \cdot \overline \varphi\big|_{Z(I^2(Z))}=0$, where $\overline f=(\overline f_j)_{j=1,\dots r}$, and
$\overline \varphi$ is the image of $\varphi$ in $i^*(I(Z)/I(Z)^2)$.
The contradiction finishes the proof.
\end{proof}

\begin{lemma}\label{lm:SmSurGlSRestriction}
Let $y=(y_i)_{i=1,\dots r}$, $y_i\in \Gamma(\PP^n_X, \cO(d_y))$ be a set of global sections on the projective space over an affine $k$-scheme $X$ such that  $Z(y\big|_{\A^n_X})=Y\amalg \hat Y\subset \A^n_X$, and $Y$ is smooth over $X$.
Let $s=(s_i)_{i=1\dots l}$, $s_i\in \Gamma(\PP^n_X\cO(d_s))$, be a set of global sections. Suppose that $Z(s)=Z\amalg \hat Z$, $Z \subset Y$ 
and let
$$
\sum\limits_{i=1,\dots l} w_{i,j} \cdot (s_i\big|_{Z(I^2(Z))}) = t_\infty^{d_w+d_s - d_y} y_j\big|_{Z(I^2(Z))}, 1\leq j\leq r,
.$$ 
where $w_{i,j}\in \Gamma(Z(I^2(Z)),\cO(d_w))$. 

Then the homomorphism $\Gamma(Z,\cO(d_s)^l)\to \Gamma(Z,\cO(d_y)^r)$
given by the matrix $A=(w_{i,j}\big|_{Z})$ is surjective.

\end{lemma}
\begin{proof}[Proof of lemma \ref{lm:SmSurGlSRestriction}]
Since $Z(s)\subset \A^n$ then $i_Z^*(\cO(d))\simeq \cO(Z)$, where $i_Z\colon Z\to \PP^n$ is the canonical inclusion.
So lemma \ref{lm:SmSurRegRestriction} implies that the homomorphism 
of sheaves $i_Z^*(\cO(d_s))^l)\to i_Z^*(\cO(d_y))^r$ defined by $A$ is surjective.

Moreover, since $Z(s)$ is closed in $\PP^n_X$, it is projective over $X$, and since $Z(s)\subset \A^n$, then it is finite over $X$. Then $Z$ is finite over the affine scheme $X$, so it is affine too. Thus the homomorphism $\Gamma(Z,\cO(d_s)^l)\to \Gamma(Z,\cO(d_y)^r)$ is surjective.
\end{proof}

\subsection{Proof of the smoothness}
\begin{proof}[Proof of the proposition \ref{prop:smoothness}.]

Denote by $\Gamma_{n}^{source}$ the affine space that rational points are
$\Gamma(\PP^n,\cO(d_e)\oplus\cO(d)^l\oplus\cO(d_b)\oplus\cO(d_w)^{rl}\oplus\cO(d_{\ome})^r\oplus\cO(d_c)^{nl})$,
where $d=3^n$, $d_e = h^e_d$, $d_b=d_c=h_d$, $d_w=h^\prime_d$, $d_{\ome}=d_w+d - d_e$. 

It follows from the definition \ref{def:Fraff} that
$\cFraff_n(Y\wT{l}) = A^{-1}_n(0)$, where 
$A$ is a regular map of the affine spaces over the base
\begin{gather*}
\begin{array}{ccccl}
A_n&\colon& \Gamma^{source}_n&\to& \Gamma^{target}_n\\ 
&&(e,s,b,w,\ome,c)&\mapsto& 
(Eq_0,Eq_{1},Eq_{2,1}\dots, Eq_{2,n-r},Eq_{3,1},\dots Eq_{3,n}, Eq_{4,1}\dots ,Eq_{4,n})
\end{array}\\
\Gamma^{target}_n= 
\Gamma(\PP^n, \cO(d_b+d)\oplus \cO(d_w+d)^{n-r}\oplus \cO(d_c+d)^n)\oplus\Gamma(\PP^{n-1},\cO(d_u)^n)
\\ 
\begin{array}{lcrclclcl}
Eq_1 &=& t_\infty^{d+d_b-2d_e}(t_\infty^{d_e} - e) e & - & \sum\limits_{i=1}^l b_i s_i &-& \sum\limits_{i=1}^{n-r} f_j y_j^2,\\
Eq_{2,j}&=& t_\infty^{d+d_w-d_y} y_j &-& \sum\limits_{i=1}^l w_{i,j} s_i &-&\ome _j e^2 &-& \sum\limits_{k=1}^{n} h_{k,j} u_k,\\
Eq_{3,k}&=& u_k t_\infty^{d_c+d-d_u} &-& \sum\limits_{i=1}^l c_{i,k} s_i 
&-& \sum\limits_{j=1}^{n-r} z_{k,j} y_j^2,\\
Eq_{4,k} &=& (u_k-t_k^{d_u})\big|_{\PP^{n-1}},& 
,
\end{array}
\end{gather*} 
where $k=1,\dots,n$, $j = 1\dots n-r$.

Consider the differential homomorphism $\der A_{n}\colon T_{\Gamma^{source}_{n}}\to A_n^*(T_{\Gamma^{target}_{n}})$.
The claim is to prove that $\der A_{n}$ is surjective.
This is provided by def. \ref{def:hdHpd} and lemma \ref{lm:SmSurGlSRestriction}.
Let's write what is $\der A_{n,X}$ precisely:
\begin{gather*}
\begin{array}{ccccl}
\der A_{n}&\colon& T_{\Gamma^{source}_n} &\to& A_n^*(T_{\Gamma^{target}_n})\\ 
&& (\der e,\der s,\der b,\der w,\der \ome,\der c)
&\mapsto& 
(\der Eq_0, \der Eq_{1},\dots \der Eq_{2,j}\dots, \der Eq_{3,k},\dots , \der Eq_{4,k}\dots )
\end{array}\\
\begin{array}{lcl}
T_{\Gamma^{source}_n} &\simeq& 
\Gamma(\PP^n, \cO(d_e)\oplus\cO(d)^l\oplus\cO(d_b)^l\oplus\cO(d_w)^{rl}\oplus\cO(d_\ome)^r\oplus\cO(d_c)^{nl}),\\
A_n^*(T_{\Gamma^{target}_n} )&\simeq& 
\Gamma(\hat Y,\cO(d_e))\oplus\Gamma(\PP^n, \cO(d+d_b)\oplus \cO(d_w+d)^r\oplus \cO(d_u)^n )\oplus \Gamma(\PP^n, \cO(d_u)^n),
\end{array}\\ 
\begin{array}{lcrclclclcl}
\der Eq_1 &=& t_\infty^{d+d_b-2d_e}(t_\infty^{d_e} - 2e)\der e & - & \sum\limits_{i=1}^{n+l} (\der b_i s_i + b_i \der s_i), &-& \sum\limits_{i=1}^{n-r} \der f_j y_j^2,
\\
\der Eq_{2,j}&=& t_\infty^{d+d_w-d_y} \der y_j &-& \sum\limits_{i=1}^l (\der w_{i,j} s_i + w_{i,j} \der s_i) &-& \sum\limits_{k=1}^{n} (\der h_{k,j} u_k + h_{k,j} \der u_k) &-& \\
&&&-&\ome_j 2e\der e - \der\ome_j e^2&-&\sum\limits_{j=1}^{n-r} \der z_{k,j} y_j^2,
\\
\der Eq_{3,k}&=& t_\infty^{d_c+d-d_u} \der u_k &-& \sum\limits_{i=1}^l (\der c_{i,k} s_i + c_{i,k} \der s_i), 
\\
\der Eq_{4,k} &=&\der u_k\big|_{\PP^{n-1}},& 
\end{array}
\end{gather*}

Then results of lemmas \ref{lm:surEq1},\ref{lm:surEq2}, \ref{lm:surEq34}, which follow in the text, give us the surjections
\begin{equation}\label{eq:surEq}
\begin{array}{rcl}
\Gamma(\PP^n,\cO(d_e)\oplus \cO(d_b)^{n+l}\oplus \cO(d_f)^{n-r})&\twoheadrightarrow& 
\Gamma(\PP^n,\cO(d+d_b))\colon \\
(\der e,\der b)&\mapsto& (t_\infty^{d+d_b-2d_e}(t_\infty^{d_e} - 2e)\der e - \sum\limits_{i=1}^{n+l} \der b_i s_i - \sum\limits_{i=1}^{n-r} \der f_j y_j^2) \\
\Gamma(\PP^n,\cO(d)^{n+l}\oplus \cO(d_w)^{r(n+l)}\oplus \cO(d_\ome)^{r})&\twoheadrightarrow& \Gamma(\PP^n,\cO(d+d_w)^r)\colon \\
(\der s,\der w,\der \ome)&\mapsto& (\sum\limits_{i=1}^{n+l} (\der w_{i,j} s_i+w_{i,j} \der s_i) + \sum\limits_{k=1}^{n} \der h_{k,j} u_k + \der\ome_j e^2 )_{j=1,\dots, r} \\
\Gamma(\PP^n,\cO(d_u)^n\oplus \cO(d_c)^{n(n+l)})&\twoheadrightarrow& \Gamma(\PP^n,\cO(d_u)^n)\oplus \Gamma(\PP^n,\cO(d_u)^n)\colon \\
(\der u,\der c)&\mapsto& (\der u_k-\sum\limits_{i=1}^{n+l} \der c_{i,k} s_i-\sum\limits_{j=1}^{n-r} \der z_{k,j} y_j^2, \der u_k\big|_{\PP^{n-1}} )_{k=1\dots n}
\end{array}
\end{equation}
for any $a\in \cFraff(Y\wT{l})$.

Summing this surjections together we prove surjectivity of $\der A_{n}$. 

Consider the tangent sheaf (module) $T_{\cFrsaf_{n^\prime}}=\Ker(dA_{n})$.
By the above we get that
$(T_{\cFrsaf_{n}})$ 

 is a free coherent sheaf on $\cFraff_{d,n}(Y\wT{l})$ of the rank $dim_{\cFraff_n(Y\wT{l})}$, which yields that $\cFraff_{d,n}(Y\wT{l})$ is smooth. 

\newcommand{\bfZ}{\mathbf{Z}}
So all what we need is to prove surjectivity of the homomorphisms \ref{eq:surEq}.
Denote $X= \cFraff_{n,d}(Y\wT{l})$, $p_{n,X}\colon \PP^n\to X$.
Let $\bfa\in \Fraff(X,Y\wT{l})$ be the universal section, and denote $\bfZ = Z(\bfe,\bfs,\bfu)$.
Denote $\ovovbfs = \bfs\big|_{X\times Z^{(y^2)}}$, $\ovovbfe = \bfe\big|_{X\times Z^{(y^2)}}$,
$\ovovbfu = u\big|_{X\times Z^{(y^2)}}$,
where
$Z^{(y^2)}=Z(y_1^2,\dots y_{n-r}^2)$.

\begin{lemma}\label{lm:surEq1}
For any $a\in \Fraff(X,Y\wT{l})$, the homomorphism 
\begin{equation}\label{eq:sur1}
\begin{array}{ccc}(p_{n,X})_*(\cO(d_e)\oplus \cO(d_b)^{n+l}\oplus \cO(d_f)^{n-r})&\to& (p_{n,X})_*(\cO(d+d_b))\colon \\
(\der e,\der b)&\mapsto& t_\infty^{d+d_b-2d_e}(t_\infty^{d_e} - 2\bfe)\der e - \sum\limits_{i=1}^{n+l} \der b_i \bfs_i - \sum\limits_{i=1}^{n-r} \der f_j y_j^2 \end{array}
\end{equation}
is surjective.
\end{lemma}
\begin{proof}

It follows from lemma \ref{lm:FrafftpFr} that 
$Z(\ovovbfs)\subset X\times \A^n$, and $(\bfe^2-\bfe)\big|_{Z(\ovovbfs)}=0$.
Hence $t_\infty^{d+d_b-2d_e}(t_\infty^{d_e} - 2\bfe)$ is invertible on $Z(\ovovbfs)$.
Thus by definition \ref{def:hdHpd}
the homomorphisms 
\begin{gather*}
\Gamma(\PP^n_X, \cO(d_e))\to \Gamma(Z(\ovovbfs),\cO(d_b+d)) \colon \der e\mapsto (t_\infty^{d+d_b-2d_e}(t_\infty^{d_e} - 2\bfe)\der e)\big|_{\bfZ}
\\
\begin{array}{rcl}
\Gamma(\PP^n_{X\times P},\cO(d_b)^{n+l}\oplus \cO(d_f)^{n-r} )&\to& 
\Gamma(\PP^n_{X\times P},\cI(Z(\ovovbfs))(d_b+d)) \colon\\ 
(\der b,\der f)&\mapsto& \sum_{i=1}^l \der b_i \bfs_i + \sum_{j=1}^{n-r} \der f_j y_j^2,
\end{array}
\end{gather*}
are surjective.
Now the exact sequence
$$0\to \Gamma(\PP^n_X,\cI(Z(\ovovbfs))(d+d_b))\to \Gamma(\PP^n_X,\cO(d+d_b)) \to \Gamma(Z(\ovovbfs),\cO(d+d_b))
$$ 
yields that the homomorphism \eqref{eq:sur1} is surjective.
\end{proof}
\begin{lemma}\label{lm:surEq2}
The homomorphism 
\begin{equation}\label{eq:sur2}
\begin{array}{c}
\Gamma(\PP^n,\cO(d)^{n+l}\oplus \cO(d_w)^{r(n+l)}\oplus \cO(d_\ome)^{r})\twoheadrightarrow \Gamma(\PP^n,\cO(d+d_w)^r)\colon \\
(\der s,\der w,\der \ome)\mapsto (\sum\limits_{i=1}^{n+l} (\der w_{i,j} s_i+w_{i,j} \der s_i) + \sum\limits_{k=1}^{n} \der h_{k,j} u_k + \der\ome_j e^2 )_{j=1,\dots, r} 
\end{array}\end{equation}
is sujective.
\end{lemma}\begin{proof}
By
def. \ref{def:hdHpd} we see that
the homomorphism 
$$
\Gamma(\PP^n_X,\cO(d_w)^{n+l}\oplus\cO(d_\ome))\to \Gamma(\PP^n_X,\cI(\bfZ)(d_w+d)) \colon 
(\der w,\der \ome)\mapsto \sum_{i=1}^{n+l} \der w_{i,j} \bfs_i + \der\ome_j \bfe
$$
is surjective. 
Now by lemma \ref{lm:SmSurGlSRestriction} the homomorphism 
$$\Gamma(\PP^n_X,\cO(d)^{n+l}) \to \Gamma(\bfZ, \cO(d+d_w)) \colon (\der s) \mapsto \sum_{i=1}^{n+l} w_{i,j} \der s_i
$$
is surjective.
Whence 
the exact sequence 
$$
0\to \Gamma(\PP^n_X,\cI(\bfZ)(d+d_w))\to \Gamma(\PP^n_X,\cO(d+d_w)) \to \Gamma(\bfZ,\cO(d+d_w))
$$
yields that the homomorphism \eqref{eq:sur2} is surjective.
\end{proof}
\begin{lemma}\label{lm:surEq34}
The homomorphism 
\begin{equation}\label{eq:sur34}
\begin{array}{ccc}
\Gamma(\PP^n,\cO(d_u)^n\oplus \cO(d_c)^{n(n+l)})&\twoheadrightarrow& \Gamma(\PP^n,\cO(d_u)^n)\oplus \Gamma(\PP^n,\cO(d_u)^n)\colon \\
(\der u,\der c)&\mapsto& (\der u_k-\sum\limits_{i=1}^{n+l} \der c_{i,k} s_i-\sum\limits_{j=1}^{n-r} \der z_{k,j} y_j^2, \der u_k\big|_{\PP^{n-1}} )_{k=1\dots n}
 \end{array}\end{equation}
is surjective.
\end{lemma}
\begin{proof}
By definition we have
$Z(\ovovbfs)=Z(\bfs,y_1^2,\dots y_{n-r}^2)$,
and 
it follows from lemma \ref{lm:FrafftpFr} that 
$Z(\ovovbfs)\subset X\times \A^n$.
According to def. \ref{def:hdHpd} we have surjections
$$
\begin{array}{rcl}
\Gamma(\PP^n_{X\times P},\cO(d_c)^{n+l}\oplus \cO(d_z)^{n-r} )&\to& 
\Gamma(\PP^n_{X\times P},\cI(Z(\ovovbfs))(d_c+d)) \colon\\ 
(\der c,\der z)&\mapsto& \sum_{i=1}^l \der c_i \bfs_i + \sum_{j=1}^{n-r} \der z_j y_j^2,
\\
\Gamma(\PP^n_X,\cO(d_u)) &\twoheadrightarrow& \Gamma(Z(\ovovbfs)\amalg \PP^{n-1}_X,\cO(d_u)) \colon\\
(\der u_k)&\mapsto& (u_k\big|_{Z(\ovovbfs)},u_k\big|_{\PP^{n-1}_X}) 
\end{array}
$$ 
where $k=1,\dots n$.
Hence homomorphism \eqref{eq:sur34} is surjective because of the exact sequence
$$
0\to \Gamma(\PP^n_X,\cI(Z(\ovovbfs))(d_u))\to \Gamma(\PP^n_X,\cO(d_u)) \to \Gamma(Z(\ovovbfs),\cO(d_u)).
$$
\end{proof}

\end{proof}

\section{Models for $M_{\PP^1}$ and $M^{\Gm}_{fr}$ via $Fr^{qaf}$.}

Lemma \ref{lm:FrafftpFr} yields that there is a natural map
\begin{equation}\label{eq:FrafftoFr}
\begin{array}{ccl}
\Fraff_n(-\times P,Y/U\wT{l})&\to& Fr_n(-\times P,Y/U\wT{l})\\
(e, s, u, w, \ome , c)&\mapsto&(Z(e,s), V, s_1/t_\infty^{d},\dots s_n/t_\infty^{d}, s_{l+1}/t_\infty^{d}, \dots s_l/t_\infty^{d}, pr_{Y})\\
&&V=\A^{n}_{-\times P}- Z(s)-Z(e,s)
\end{array}\end{equation}
see \eqref{eq:Freseqsmet} for $pr_{Y}$.
Hence there is a natural map $ \Fraff_n(-\times P,Y/U\wT{l})\to Fr^\text{nr}_n(-\times P,Y/U\wT{l})$.
The map \eqref{eq:FrafftoFr} is not agreed with the stabilisation according to the definition \ref{def:stabFraff}, 
but the map $\Fraff_n(-\times P,Y/U\wT{l})\to Fr^\text{nr}_n(-\times P,Y/U\wT{l})$ is agreed.
So we get the natural map
\begin{equation}\label{eq:FrafftoFrnr}
\Fraff(-\times P,Y/U\wT{l})\to Fr^\text{nr}(-\times P,Y/U\wT{l}).
\end{equation}
We are going to prove that this is an $\A^1$-Nis-equivalence.
The morphism is the composition of the following
\begin{multline}\label{eq:FrafftospltoFrnr}
\Fraff(-\times P,Y/U\wT{l})\to \Fragc(-\times P,Y/U\wT{l})\to Fr^\text{p-agc}(-\times P,Y/U\wT{l})\to\\
Fr^\text{nr-c}(-\times P,Y/U\wT{l})\to Fr^\text{nr}(-\times P,Y/U\wT{l}).
\end{multline}
according to the following definitions
\begin{definition}
$Fr^\text{p-agc}(-\times P,Y/U\wT{l})=\varinjlim_n Fr^\text{p-agc}_n(-\times P,Y/U\wT{l})$ are pointed presheaves 
with sections $Fr^\text{p-agc}_n(X\times P,Y/U\wT{l})$ given by
the sets 
$(e,s)$ with $e$ and $s$ like as in \eqref{eq:defFraff} and $Y,U$ are under context \ref{context:Y0}.

$\Fragc(-\times P,Y/U\wT{l})\subset Fr^\text{p-agc}(-\times P,Y/U\wT{l})$, $(e,s)\in \Fragc(X\times P,Y/U\wT{l})$ iff $Z(s)\cap (X\times \overline{\hat Y})=Z(s)\cap (\overline Y\setminus Y)=\emptyset$.
\end{definition}
\begin{definition} 
$ Fr^\text{nr-c}(-\times P,Y\wT{l}) = \varinjlim_n Fr^\text{nr-c}_n(-\times P,Y\wT{l})$, 
$Fr^\text{nr-c}_n(-\times P,Y\wT{l})\subset Fr^\text{nr}_n(-\times P,Y\wT{l})$, 
$(Z,\tau,g)\in Fr^\text{nr-c}_n(-\times P,Y\wT{l})$ iff 
$Z\subset X\times P\times Y\times 0\times X\times P\times \A^N\times\A^{n-N}$, $g=pr^n_N$ is given by the projection to the first $N$ coordinates, 
see context \ref{context:Y0} for $N$.
\end{definition}

\begin{proposition}\label{prop:EquivFraffFr}
Under the context \ref{context:YU0oc} the natural map of presheaves
\eqref{eq:FrafftospltoFrnr}
is $\A^1$-Nis-equivalence.
\end{proposition}
\begin{proof}
1)
The last morphism restricted to the category of affine schemes is an $\A^1$-equivalence by lemma \ref{lm:HomEq}. 
Actually there are morphisms of presheaves 
$$
\begin{array}{cccccccccccc}
r_n(X)&\colon& 
Fr^{\text{c-nr}}_n(Y\wT{l})(X)&\to& Fr_n^{\mathrm{nr}}(Y\wT{l})(X)&\colon  &
(Z,\phi,\psi)&\mapsto &(Z,\phi,\psi,pr^n_N)\\
l(X)&\colon& Fr^{\mathrm{nr}}_n(Y\wT{l})(X)&\to& Fr_n^{\text{c-nr}}(Y\wT{l})(X)&\colon& (Z,\phi,\phi,g)&\mapsto&  (\Gamma_g, \gamma_1,\dots \gamma_N,\phi,\psi),
\end{array}
$$
where $\Gamma_g$ is a graph of the morphism $g\colon Z\to Y$
considered as a subset in $\A^N\times \A^n_X$ via inclusions $Y\to \A^N$, $Z\to \A^n_X$, see context \ref{context:Y0},
and $\gamma_i = t_i - w_i\circ \tilde g$ where $w_i$ denotes coordinates on $\A^N$ and $\tilde g\colon \A^n_X\to \A^N$ is a lift of $g$.
To get the claim we need to construct $\A^1$-homotopies 
$$
r\circ l\stackrel{h_l}{\sim} \sigma^N_{Fr^{\mathrm{nr}}_n(Y\wT{l})}\colon Fr^{\text{c-nr}}_n(X) \to Fr^{\text{c-nr}}_{n+N}(\A^1\times X),\;
l\circ r \stackrel{h_r}{\sim} \sigma^N_{Fr^{\text{c-nr}}(Y\wT{l})}\colon Fr^{\mathrm{nr}}_n(X) \to Fr^{\mathrm{nr}}_{n+N}(\A^1\times X)
$$
To get $h_l$
consider the homotopy 
$h^\prime_l(X)\colon Fr^{\mathrm{nr}}_n(X)\to Fr^{\mathrm{nr}}_{n+N}(\A^1\times X)\colon (Z,\phi,\psi,g)\mapsto
\tilde\Gamma_g,\tilde\gamma,\phi,\psi,g),$
where 
$\tilde\Gamma_g = Z(\tilde\gamma)\cap (\A^N\times Z)$, $\tilde\gamma=(\tilde\gamma_i)_{i=1\dots N}$, $\tilde\gamma_i=t_i - \lambda (w_i\circ \tilde g)\in \cO(\A^1\times\A^{N+n}_X)$.
Then $h^\prime_l(X)$ connects $r\circ l\sim P\circ\sigma^N$, 
where $P$ is an endomorphism on $Fr^{norm}_{n+N}(Y\wT{l})$ defined by the automorphism of $\A^{n+N}$ given by permutation of coordinates $(t_1,\dots t_{n+N})\mapsto (t_{N+1},\dots t_n,t_1,\dots ,t_N)$.
Since $N$ is even according to context \ref{context:Y0}, $P$ is $\A^1$-homotopy equivalent to the identity.
The second homotopy $h_r$ is given in a similar way as a composition of the homotopy 
$h^\prime_r\colon (Z,\tilde\gamma, \phi,\psi)\mapsto (\tilde\Gamma_g,\tilde\gamma, \phi,\psi)$, where $g=pr^n_N\big|_{Z}$,
and the permutation $P$. 

2)
The second last morphism is an $\A^1$-Nis-equivalence by proposition \ref{lm:Freq1}
because of the isomorphism on presheaves $ \Fragc= Fr^{eq}$ on affines.
The second morphism in \eqref{eq:FrafftospltoFrnr} is an $\A^1$-Nis equivalence, since the condition $Z(s)\cap (X\times \overline{\hat Y})=\emptyset$ is a condition of infinite codimension. 

3)
The first morphism in \eqref{eq:FrafftospltoFrnr} is $\A^1$-Nis-equivalence by proposition \ref{prop:LiftCriteria}.
Let us skip the checking of the closed glueing that is straightforward.
The lifting property for the section $u$ follows directly from lemma \ref{lm:lift} presented in Appendix A;
The lifting property for the rest data follows from def. \ref{def:hdHpd}, lemma \ref{lm:StSurIdeals}, corollary \ref{cor:SerresTh}, and Lm-remark \ref{rm:Resincdim} in Appendix A;
\end{proof}

\begin{corollary}
Under context \ref{context:YU0oc}
the pointed sheave $Fr(-\times P,Y/U\wT{l})$ are $\A^1$-Nis-equivalent to the factor sheave represented by the pair of ind-schemes $\cFraff(P,Y/U\wT{l})$.
\end{corollary}

Define the $T$-spectrum
$$M^\prime_{T}(Y) = (\Fraff(Y), \dots \Fraff(Y\wT{l}))$$
with  
$$\begin{array}{lcl}
\Fraff(Y\wT{l})\wedge T&\to& \Fraff(Y\wT{l+1})\\
((e,s, u, c, b, w, \ome, f, h, z, \oveb, \ovec,\ovew), x)&\mapsto & (e,s, t_\infty^{3^n} x, u, c, b, w, \ome, f, h, z, \oveb, \ovec,\ovew)
\end{array}$$
Let $M^\prime_{\PP^1}(Y)$ denotes the $\PP^1$-spectrum obtained from $M^\prime_{T}(Y)$ using the standard morphism of pointed sheaves $(\PP^1\infty)\to (\PP^1,\PP^1-0)\simeq T$.
Now proposition \ref{prop:EquivFraffFr} implies
\begin{corollary}\label{cor:MqafPP1}
For any $Y$ under context \ref{context:Y0}
the canonical morphism of spectra $M^\prime_{\PP^1}(Y)\to M_{\PP^1}(Y)$ is an equivalence. 
\end{corollary}
\begin{theorem}
Let $Y\in \Sm_S$ over a base $S$ and let $Y$ be affine over $S$ or $S=\Spec k$ for a regular noetherian ring $k$. 
Then there is a $T$-spectrum $M^\prime_{T}(Y)$
in the category $\text{ind-}\Sm^text{o-pairs}$
with a section wise motivic equivalences of the $\PP^1$-spectra of motivic spaces
$M^\prime_{\PP^1}(Y)\to M_{\PP^1}(Y)$,
where
$M^\prime_{\PP^1}(Y)$ is the $\PP^1$-spectrum defined by $M^\prime_{T}(Y)$.
\end{theorem}
\begin{proof}
The case of an arbitrary $Y\in \Sm_S$ can be reduced to the case of 
affine smooth $Y$
by Jouanolou-Tomason's trick \cite[Theorem 1.1]{Asok-JTt}.  
Next the case of affine smooth $Y$ can be reduced to the case of an affine scheme with the trivial normal bundle
\ref{lm:normbundlrepl}.
Thus since any smooth affine scheme with trivial normal bundle fits into context \ref{context:Y0} the claim follows from theorem \ref{th:posomegaP1motfibrepl} and corollary \ref{cor:MqafPP1} 
\end{proof}

In a similar way to the spectrum $M_{fr}(Y)$ in \cite[section 11]{GP14} we can define the spectrum $M^{\prime\prime}_{\Gm}(Y/U)$ using the sheaves $\Fraff(Y\wT{l})/\Fraff(U\wT{l})$.
Now for the case of quasi-projective smooth $Y$ and an open $U\subset Y$ and for an arbitrary smooth $Y$ and closed smooth $U\subset Y$ we have the section-wise equivalence $M^{\prime}_{\Gm}(Y/U)\simeq M^{\prime\prime}_{\Gm}(Y/U)$. 
In the same time the following proposition follows straightforward from the definition.
\begin{proposition}\label{prop:MprimeG}
$M^\prime_{\Gm}(Y/U)_f$ is termwise equivalent to  $M^{\Gm}_{fr}(Y/U)$ in positive degrees in $S^1$-direction, where $M^{\Gm}_{fr}(Y/U)$ is defined like as in \cite[section 11]{GP14} using $Fr(Y\wT{l})$.
\end{proposition}
Finally we have the following representability result
\begin{theorem}
The bi-spectrum $M^{\prime\prime}_{\Gm}(Y/U)$ is stably motivically equivalent to $\Sigma^\infty_{\Gm}\Sigma^\infty_{S^1}(Y/U)$, the termwise fibrant replacement with respect to the injective (Nisnevich) local model structure $M^{\prime\prime}_{\Gm}(Y/U)$ is 
motivicaly fibrant $\Omega$-bi-spectrum in positive degrees with respect to $S^1$-direction.\end{theorem}
\begin{proof}
The case of $U=\emptyset$ follows form the above proposition and that fact that $M^{\prime\prime}_{\Gm}(Y)=M^\prime_{\Gm}(Y)$.
In general the equivalence $M^{\prime\prime}_{\Gm}(Y/U)\simeq \Sigma^\infty_{\Gm}\Sigma^\infty_{S^1}(Y/U)$ follows form the ones for $M^{\prime\prime}_{\Gm}(Y)$ and $M^{\prime\prime}_{\Gm}(U)$. That fact that $M^{\prime\prime}_{\Gm}(Y/U)$ is motivically fibrant in positive degrees follows by the similar arguments that are used for $M^{\Gm}_{fr}$ in \cite[section 11]{GP14}.
\end{proof}

\section{Appendix A: Lifting properties for sections of coherent sheaves.}

In the appendix we summarise some results on coherent sheaves used in the article.

\begin{lemma}\label{lm:nondegsup}
Let $f\colon V\to F$ be a homomorphism of coherent sheaves on a scheme $X$, and $V$ be locally free of a finite rank.
Then the set of points $x\in X$ such that $\Coker(i_x^*(f))=0$ is closed, where $i_x\colon x\hookrightarrow X$.
\end{lemma}
\begin{proof}
Since the question is local we can assume that $X$ is affine.
Then any coherent sheaf $F$ on $X$ can be represents as a cokernel of a locally free coherent sheaves of a finite rank,
and we can assume that $F$ is locally free of a finite rank without lose of generality.
Next since $\Coker(i_x^*(f))=0$ iff $\Coker(\bigwedge^{r} f)=0$, $r=\rank F$, 
we can assume that $F$ is an invertible sheaf. 
Consider the dual morphism $f^*\colon D(F)\to D(V)$, $D(F)=\mathcal Hom(F,\cO(X))$, $D(V)=\mathcal Hom(V,\cO(X))$. 
Then $\{x\in X\colon \Coker(i^*_x(f))=0\}=\Supp(\Ker f^*)$.
\end{proof}

The rest part of the Appendix is about consequences of Serre's theorem on an ample bundles and cohomologies of coherent sheaves.
Let us recall the theorem.
\begin{theorem}[Serre's theorem]
Let $F$ be a coherent sheaf on a scheme $X$ and $\cO(1)$ be an ample bundle.
Then for some $N\in \mathbb Z$, for all $d>N$ the cohomologies presheaves of $F(d)=F\otimes\cO(d)$ are trivial.
\end{theorem}
We use this theorem in the following form:
\begin{corollary}
Let $F\twoheadrightarrow G$ be a surjective morphism of coherent sheaves on a scheme $X$ with an ample bundle $\cO(1)$; then for all large enough $d\in \mathbb Z$ the homomorphism of global sections $\Gamma(X,F(d))\to \Gamma(X,G(d))$ is surjective.
\end{corollary}
Let us also formulate the following particular case
\begin{corollary}\label{cor:SerresTh}
Let $X^\prime\hookrightarrow X$ be a closed embedding, and let $\cO(1)$ be an ample bundle on $X$; let $F$ be a coherent sheaf of $F$. Then for all large enough $d$ the restriction $\Gamma(X,F(d))\to \Gamma(X^\prime,F(d))$ is surjective, where $F(d)=F\otimes \cO(d)$.
\end{corollary}

\begin{lemma}\label{lm:StSurIdeals}

Let $Y$ be a projective scheme over some base $X$.
Let $s_i\in \Gamma(Y,\mathcal O(d_i))$, $i=1,\dots l$, 
Denote $I(Z)(d) = \Gamma(\PP^n_X, \cI(Z)(d)) = \{s\in \Gamma(Y,\mathcal O(d))| s\big|_Z=0\}$.

Then $\exists N\in \mathbb Z$ such that $\forall d>N$ the map
$$\bigoplus \Gamma(Y,\mathcal O(d-d_i))\to I(Z)(d)\colon\;
(\alpha_1,\dots \alpha_l)\mapsto  \sum\limits_{i=1,\dots, n}s_i\alpha_i$$
is surjective.
\end{lemma}
\begin{proof}

Consider the homomorphism of coherent sheaves $e\colon \bigoplus_{i=1,\dots l}\cO(-d_i)\to \cI(Z)$ given by the vector $s=(s_i)_{i=1,\dots ,l}$, where $\cI(Z)$ denotes the sheaf of ideals coresponding to the closed subscheme $Z=Z(s)$. Then $e$ is surjective, and $e(d)\colon \sum_{i=1,\dots l}\cO(d-d_i)\to \cI(Z)(d)$ is surjective for an large enough $d$.
\end{proof}

\begin{lemma}\label{lm:liftzero}
Let $e\colon X^\prime\hookrightarrow X$ be a closed embedding of affine schemes.

Then for all $d\in \mathbb Z$
for any sections
$s_i\in \GlSX{n}{X}{d_{i}}$,
$w^\prime_i\in \GlSX{n}{X^\prime}{d-d_i}$, $i=1,\dots l$,
such that
$$ \sum_{i=1,\dots l} w^\prime_{i} e^*(s_i) =0, $$
there is a vector of sections $w=(w_i)$, $w_i\in \Gamma(\PP_{X^\prime},\cO(d-d_i))$, $i=1,\dots l$,
such that $$\sum_{i=1,\dots l} w_i s_i=0, \; e^*(w_i)=w^\prime_i=0.$$
\end{lemma}
\begin{proof}
Consider the morphism of coherent sheaves $h\colon \bigoplus \cO(d-d_i)\to \cO(d)\colon (w_i)\mapsto \sum w_i s_i$, and
denote $\mathcal E=\Ker h(d)$. 
Then $w^\prime\in \Gamma(\PP^n_{X^\prime},\mathcal E)=\Gamma(X, p_*(\mathcal E))$, where $p\colon \PP^n_X$ is the canonical projection.
Now since the direct image $p_*(\mathcal E)$ is a coherent sheaf on the affine scheme,
it follows that $\exists w\in \Gamma(X, p_*(\mathcal E))$, $e^*(w)=w^\prime$.
This finishes the proof.
\end{proof}

\begin{lemma}\label{lm:lift}
Let $e\colon X^\prime\hookrightarrow X$ be a closed embedding of affine schemes.
Let
$s_i\in \GlSX{n}{X}{d_{i}}$, $i=1\dots l$,
and assume that $d\in \mathbb Z$
is such that the homomorphism $\bigoplus\limits_{i=1}^l \cO(d-d_i)\to \cI(Z(s))(d)\colon (w_1,\dots w_l)\mapsto \sum\limits_{i=1}^l w_i s_i$ is surjective.

Then for any sections
$a\in \GlSX{n}{X}{d}$,
$w^\prime_i\in \GlSX{n}{X^\prime}{d-d_i}$, $i=1,\dots l$,
 such that
$$ e^*(a) = \sum_{i=1,\dots l} w^\prime_{i} e^*(s_i), \; a\big|_{Z(s_1,\dots s_l)} = 0,$$
there is a vector of sections $w=(w_i)$, $w_i\in \Gamma(\PP_{X^\prime},\cO(d-d_i))$, $i=1,\dots l$,
such that $$a = \sum_{i=1,\dots l} w_i s_i, \; e^*(w_i)=w^\prime_i.$$
\end{lemma}
\begin{proof}
By assumption on $d$ there is some section $\tilde{w}_i\in\Gamma(\PP_{X^\prime},\cO(d-d_i))$, $i=1,\dots l$, such that $a = \sum_{i=1,\dots l} \tilde w_i s_i$.
Now the claim follows form lemma \ref{lm:liftzero} applied to $\tilde w - w^\prime$, where $\tilde w= (\tilde w_i)_{i=1,\dots l}$, $w^\prime= (w_i^\prime)_{i=1,\dots l}$.
\end{proof}
\begin{remark}
lemma \ref{lm:lift} and lemma \ref{lm:StSurIdeals}
implies the following result:

For any closed embedding of affine schemes $e\colon X^\prime\hookrightarrow X$,
$\exists D\in \mathbb Z,\forall d>D$,
for any sections
$a\in \GlSX{n}{X}{d}$,
$s_i\in \GlSX{n}{X}{d_{i}}$,
$w^\prime_i\in \GlSX{n}{X^\prime}{d-d_i}$, $i=1,\dots l$,
 such that
$$ e^*(a) = \sum_{i=1,\dots l} w^\prime_{i} e^*(s_i), \; a\big|_{Z(s_1,\dots s_l)} = 0,$$
there is a vector of sections $w=(w_i)$, $w_i\in \Gamma(\PP_{X^\prime},\cO(d-d_i))$, $i=1,\dots l$,
such that $a = \sum_{i=1,\dots l} w_i s_i, \; e^*(w_i)=w^\prime_i.$
\end{remark}

\begin{lemma-remark}\label{rm:Resincdim}
For any affine $X$ and $n,d\in \mathbb Z$ 
elements of $\Gamma(\PP^{n_1},\cO(d))$ are homogeneous polynomials of degree $d$ with coefficients in $\cO(X)$.
Since any polynomial of $n$ variables can be considered as a polynomial of larger amount of variables we get the following:
\end{lemma-remark}
\begin{proof}
For any affine $X$ and integers $0<n_1<n_2$, and $d\in \mathbb Z$ 
the restriction homomorphism $\Gamma(\PP^{n_2},\cO(d))\to \Gamma(\PP^{n_1},\cO(d))$ is surjective.
Moreover there is a canonical homomorphism $\Gamma(\PP^{n_1},\cO(d))\to\Gamma(\PP^{n_2},\cO(d))$ that is left inverse to the restriction $\Gamma(\PP^{n_2},\cO(d))\to \Gamma(\PP^{n_1},\cO(d))$.
\end{proof}
\begin{lemma-remark}\label{rm:Resincd}
Let $X$ be  an affine scheme, $d\in \mathbb Z$, and $Z\subset \A^n_X$ be a closed subscheme finite over $X$.
Let $t_\infty\in \GlS{n}{1}$, $Z(t_\infty)=\PP^{n-1}=\PP^n\setminus \A^n_X$.
Consider the restriction homomorphisms $f_{d+1}\colon \Gamma(\PP^n_X,\cO(d+1))\to \Gamma(Z,\cO(d+1))$,
$f_{d}\colon \Gamma(\PP^n_X,\cO(d))\to \Gamma(Z,\cO(d))$.
Then $\Image(f_d)\subset \Image(f_{d+1})$.
\end{lemma-remark}
\begin{proof}
Actually,
$\Gamma(Z,\cO(d+1))\simeq \Gamma(Z,\cO(d))\simeq cO(Z)$ where the
isomorphisms are defined by the multiplication by $t_\infty$ and $t_\infty^d$.
So the homomorphism $t_\infty\colon \GlSX{n}{X}{d}\to \GlSX{n}{X}{d}$ induced by the multiplication by $t_\infty$
induces the homomorphism $\Image(f_{d})\to \Image(f_{d+1})$. 
\end{proof}

\section{Appendix B: $\A^1$-Nis-equivalences}
To prove $\A^1$-Nis equivalences we use three following criteria.
where the first one is straightforward,
and the second one is contained inside the proofs form 
\cite{EHKSY-infloopsp}.

\begin{lemma}\label{lm:HomEq}
Let $f\colon F\to G$ and $g\colon G\to F$ be a pair of morphisms of presheaves on $Sm_k$,
and let 
$h_F\colon F\to F^{\A^1}$, $F^{i_0}\circ h_F = g \circ f$, $h_F^{i_1}=id_F$,
$h_G\colon G\to G^{\A^1}$, $G^{i_0}\circ h_G = f \circ g$, $h_G^{i_1}=id_G$,
where $F^{\A^1}(-)=F(-\times\A^1)$, $F^{i_0}(-)=i_0^*(-)$, $i_0\colon 0\to \A^1$, and similarly for $G$ and $i_1$.
\end{lemma}\begin{proof}
\end{proof}
\begin{proposition}\label{prop:LiftCriteria}
Let $e\colon F\to G$ be a morphism of presheaves on $Sm_k$.
Suppose that $e$ satisfies the lifting property with respect to closed embeddings of affines, and both presheaves $F$ and $G$ satisfy closed glueing; then $e$ is $\A^1$-Nisnevich equivalence.
\end{proposition}\begin{proof}
The proof is contained in the proof of \cite[proposition 2.2.21]{EHKSY-infloopsp}.
Let us briefly recall it. 
The lifting property with respect to closed embedding of affines imply that the morphism is surjective on affines.
Hence 
\cite[lemma A.2.6]{EHKSY-infloopsp}.
implies that
the morphism $e$ being valeted on affines is a trivial Kan fibration.
Whence, since any scheme admits an affine Zariski covering, $e$ is an $\A^1$- Nisnevich equivalence.
\end{proof}

\begin{definition}
Denote by $\Pi_n$ the union of the $\Delta^n\times 1\subset \A^{n+1}\times\A^1$ and $\delta\Delta_n\times\A^1\subset \A^{n+1}\times\A^1$.
\end{definition}
\begin{proposition}\label{prop:LiftCriteria}
Let $e\colon F\to G$ be a morphism of presheaves on $Sm_k$.
Suppose that 
both presheaves $F$ and $G$ satisfy closed glueing,
and suppose that 
for any simplicial model $\delta \hookrightarrow \Delta$ for the embedding $\delta\Delta_n\to \Delta_n$,
a morphism $v\colon \Delta\to G$, and a lift $r\colon \delta\to F$, of the morphism $\delta\to G$, there is a lift of $v$ to a morphism 
$\Delta\to F$.

Then $e$ is a Nisnevich $\A^1$ equivalence.
\end{proposition}
\begin{proof}
Denote by $I$ the model structure on the category of pointed simplicial presheaves on $Sm$ cofibrantly generated by $\A$-geometric realisations of cofibrations in the injective model structure on pointed simplicial sets, and closed embeddings and coverings. 
Then any trivial cofibration in $I$ is a Nisnevich $\A^1$ equivalence (and even $\A^1$ equivalence) in the category of pointed simplicial presheaves.

Consider the fibrant replacement $\tilde f\colon \tilde F\to \tilde G$ of $f$ with respect to a model structure $I$.
Then $G\to \tilde G$ and $F\to \tilde F$ are  Nisnevich $\A^1$ equivalences.
The claim now is to prove that $\tilde f$ is a Nisnevich $\A^1$ equivalence.

Actually check that $\tilde f$ satisfies the condition of the criteria \ref{lm:HomEq}.
Let $\tilde r\colon \delta\Delta_n\to \tilde F$ $\tilde g\colon \Delta_n\to \tilde G$.
Then there is $r\colon \delta \to F$ and $g\colon \Delta\to G$.
Then there is $\Delta\to F$ and hence $\Delta_n\to F$ since $F$ is fibrant with respect to $I$.  
\end{proof}

\begin{lemma}\label{lm:EqforNeigh}
For any Nisnevich neighbourhood $(U,Z)\to (\A^n_X,Z)$, $X\in Sm$, $Z\subset \A^n_X$ is closed, $Z$ is finite over $X$, 
for all large enough $d_i\in \mathbb Z$,
there is a refinement $(U^\prime,Z)\to (U,Z)\to (\A^n_x,Z)$, and open immersion 
$U\hookrightarrow Z(s_1,\dots s_n)\subset \A^n\times X$ for some $s_i\in \Gamma(\PP^n_X,\mathcal O(d_i))$, $s_i\big|_{\PP^{n-1}\times X} = x_i^{d_i}$.
Moreover we can assume that $Z(I^2(Z))\times 0\subset Z(s_1,\dots s_n)$.
\end{lemma}
\begin{proof}
Since any Nisnevich neighbourhood $U\to \A^n_X$ is quasi-finite by Zariski main theorem it follows that there is an open embedding $U\hookrightarrow \overline U$ and finite morphism $\overline U\to X$. 
Since any finite morphism is affine there is an embedding $\overline U\to \A^l\times \A^n_X$. 
Moreover since there is the closed embedding $Z\to U$, which is a lift of $Z\to \A^n_X$, we can choose it in a such way that the image of $Z$ in $\A^l\times\A^n_X$ is equal to $0\times Z$.

Denote $Z_2=Z(I^2(0\times Z))\times_{\A^n_X} Z\subset \A^l\times\A^n_X$.
Choose a sections $v_i\in \Gamma(\PP^l\times\A^n_X,\mathcal O(d))$, for some $d\in \mathbb Z$, $v_i\big|_{\overline U}=0$, $v\big|_{Z_2 } = x_i^d$, $v\big|_{\PP^{l-1}\times \A^n_X} = x_i^d$, $i=1\dots l$, where $x_i$ are coordinates on $\PP^l$.
And denote by $d_1$ the minimal integer such that there are $s_i\in \Gamma(\PP^{l+n}_X,\mathcal O(d_1))$  such that $s_i/t_\infty^{d_1} = v_i/x_\infty^d$ for $i=1\dots l$, where $t_\infty\in \Gamma(\PP^{l+n},\mathcal O(1))$, $Z(t_\infty) = \PP^{l+n-1}$.

Let $e\in \mathbb Z$ be any integer such that $e d>d_1$ and
let $f\in \mathbb Z[t]$ be a polynomial of degree $e$ with unit leading term and such that $f\big|_{Z(t^2)} = t$.
Then consider the morphism $F\colon \A^l\times \A^n_X\to \A^l\times \A^n_X\colon (x_1,\dots x_l, x_{l+1}\dots x_{l+n}, x)\to (f(x_1),\dots,  f(x_l), x_{l+1}\dots x_{l+n}, x)$, .
Let $\overline U^\prime = F^{-1}(\overline U)$ and let $U^\prime =F^{-1}(U)$. 
Then if follows by assumptions that the
intersection of the closure of $\overline U^\prime$ in $\PP^{l+n}_X$ with $\PP^{l+n-1}$ is contained in the subspace $\PP^{n-1}_X=Z(t_1,\dots t_l)$, where $t_i$ denotes coordinates on $\PP^{l+n}$.
Hence for all large enough $d_i\in \mathbb Z$ there are sections $s_i^\prime\in \Gamma(\PP^{l+n}_X,\mathcal O(d_i))$, $s_i^\prime\big|_{\PP^{n-1}\times X} = x_i^{d_i}$, $s_i^\prime\big|_{Z_2} = x_i^{d_i}$, $s_i^\prime\big|_{\overline U^\prime}=0$.
Then $\overline U^\prime$ is a union of some of the irreducible components of $Z(s_1^\prime,\dots s_l^\prime)$, and whence $U^\prime$ is an open subset in $Z(s_1^\prime,\dots s_l^\prime)$.

Thus sections $s_i^\prime$ satisfies all conditions except the last one. 
Finally, 
since there is a lift of $Z$ along the morphism $p\colon Z(s_1,\dots s_l)\to \A^n_X$ and $p$ is elate on the image of $Z$ under this lift, changing coordinates on $\A^l\times\A^n_X$ (relatively over $\A^n_X$) 
we can get that $Z(I^2(Z))\times 0\subset Z(s_1,\dots s_l)$.
\end{proof}

\begin{lemma}\label{lm:EqforNeigh2}
For any Nisnevich neighbourhood $(V,Z)\to (X,Z)$, $X\in Sm$, $Z\subset X$ is closed, 
there are $f_i\in \cO(\A^{m}_X)$, $i=1,\dots m$, such that
$Z(f_1,\dots f_i)\to X$ is a refinement of $V\to X$, and $Z(\phi_1,\dots \phi_m)\supset Z(I^2(Z))\times 0$.
\end{lemma}
\begin{proof}
Since the morphism $V\to X$ is quasi-finite by Zariski main theorem it can be passed throw $V\to \overline V\to X$ with $\overline  V\to X$ being finite.
So $\overline V\to X$ is affine, let $\overline V\hookrightarrow \A^{m-1}_X$ be an embedding.
Since there is a lift $Z\to \overline V$ and $\overline V\to X$ is etale on $Z$
changing coordinates on $\A^m_X$ we can get that $\overline V\supset Z(i^2(Z))\times 0$.
Choose a function $r\in \cO(\A^{m-1}_X)$ such that $r\big|_{Z(I^2(Z))}=1$, $r\big|_{\overline V\setminus V}=0$.
Then $(f_1,\dots f_{m-1}, (1+t_{m}) r -1)$ is the required set of functions.
\end{proof}

\begin{lemma}
Let 
$f\colon F\to G$ be a morphism of simplicial presheaves on $\Sm_S$,
such that $r^*(f)$ is $\A^1$-Nis-equivalence for any $r\colon X\to S$ with affine $X$.
Then $f$ is $\A^1$-Nis-equivalence.
\end{lemma}
\begin{proof}
Let $f^f\colon F^f\to G^f$ is fibrant replacement with respect to motivic model structure on the category of simplicial presheaves.
The claim is to show that for any affine $X$, $r\colon X\to S$, the morphism of simplicial sets $f^f(X)$ is an equivalence (i.e. an isomorphism).
We know that $r^*(f)$ and consequently $r^*(f^f)$ are an equivalences. Next we know that  $r^*(F^f)$ and $r^*(G^f)$ are fibrant.
Then $r^*(f^f)$ is an equivalence of simplicial sets. But $r^*(F^f)(X)=F^f(X)$ and $r^*(G^f)(X)=G^f(X)$.
\end{proof}

Next lemma we use without a full proof which requires an accurate analyse left for other works.
\begin{lemma}\label{lm:Checkcov}
For any Nisnevich covering $w\colon W\to Y$ and open subscheme $U\subset Y$ the morphism of simplicial presheaves
$Fr^*_n(\mathcal W\wT{l}) \to Fr^*_n(Y/U\wT{l})$, $*\in \{\text{fl-e,e,st-id,fl-id,Zar},\prime,\text{Nis,1th,g-nr,nr}\}$,
where 
$\mathcal W=\{\dots 
W^n/\bigcup_i p_i^{-1}(w^{-1}(U)) 
\dots W^2_Y/(W\times_Y w^{-1}(U)\cup w^{-1}(U)\times_Y W)  \rightrightarrows W/U
\}$, 
$p_i\colon W^n\to W$, $i=1\dots n$, 
is the $\check{C}$heck simplicial object in the category of pairs in $\Sm_S$ for the morphism $W/w^{-1}(U)\to Y/U$.
\end{lemma}
\begin{proof}
The proof is similar to the original case of framed correspondences $Fr(Y)$. Namely, it is enough to show that the morphism is equivalence on henselian local schemes, which follows form the lifting property of henselian local pairs with respect to Nisnevich coverings.
\end{proof}

\begin{lemma}\label{lm:normbundlrepl} 
For any smooth affine scheme $Y$ 
there are 
an integer $d$, and
sections $y_i\in \GlS{N}{d}$, $i=1\dots r$, 
such
$Z(y_1,\dots y_r)=Y^\prime\amalg \hat Y^\prime\subset \PP^N$,
and $Y^\prime$ is isomorphic to the normal vector bundle under some closed inclusion $Y\to \A^{N^\prime}$.
\end{lemma}
\begin{proof}
The target bundle of the scheme $N_{Y/\A^N}$ is trivial. Hence its normal bundle is stable trivial. So $N_{Y/\A^n}$ is a disjoint component of a complete intersection in some affine space. 
\end{proof}

\section{Appendix C: equivalences of framed corr.}\label{sect:FrEquiv}

\begin{lemma}\label{lm:Freq1}
For any $n\in \mathbb Z_{\geq 0}$, an affine smooth $Y$, and an open $U\subset Y$, there are $\A^1$-equivalences 
$Fr^{*}_n(Y/U\wT{l})\to Fr^\text{nr}_n(Y/U\wT{l}), *\in \{\text{Nis,1st}\}$.
\end{lemma}
\begin{proof}
Both equivalences follows from the criteria given in proposition \ref{prop:LiftCriteria} (used in the proof of \cite[corollary 2.2.21]{EHKSY-infloopsp}).
It is not difficult to see that all of the presheaves satisfy the closed glueing, so to get the claim we need to prove the lifting property with respect to closed embeddings of affines. 

Let $c=(Z,W,\tau,g)\in Fr^{nr}_n(X,Y/U\wT{l})$,
and let $(Z_0,V_0,\phi_0,g_0)\in Fr_{n}(X_0,Y/U\wT{l})$ is the lift of the restriction of $c$ to an element in $Fr^{nr}_n(X_0,Y/U\wT{l})$.
It follows from lemma \ref{lm:HensLiftSmooth} that
there is an etale neighbourhood $V\to \A^n_X$ of $W\amalg_{W\times_X X_0} V_0$ 
with the lift $g^\prime\colon V\to \A^l\times Y$ of $g\amalg g_0\colon W\amalg_{W\times_X X_0} V_0\to \A^l\times Y$.
Let $\phi$ be a regular map $V\to \A^n$ such that $\phi\big|_{W}=0$, the restriction $\phi\big|_{ Z(I^2(W)) }$ is defined by the trivialisation $\tau$, and $\phi\big|_{V_0}=\phi_0$.
Now $(Z,V,\phi,g)\in Fr(X,Y/U\wT{l})$ is the required lift of $c$.

Let $c=(Z,W,\tau,g)\in Fr^{nr}_n(X,Y/U\wT{l})$. 
Let $(Z_0,\phi_0,g_0)\in Fr^{1th}_{n}(X_0,Y/U\wT{l})$ is the lift of the restriction of $c$ to an element in $Fr^{nw}_n(X_0,Y/U\wT{l})$.
Similar as above by lemma \ref{lm:HensLiftSmooth}
there is an etale neighbourhood $V\to \A^n_X$ of $Z$ 
with the lift $g^\prime\colon V\to \A^l\times Y$ of $g\amalg g_0\colon W\amalg_{W\times_X X_0} (Z(I^2(Z))\times_X X_0) \to \A^l\times Y$, and a regular function $\phi$ on $V$ that is a lift of $\phi_0$ and is defined by $\tau$ on $Z(I^2(W))$.
Now the image of $(Z,V,\phi,g)$ under the map $Fr_n(X,Y/U\wT{l})\to Fr^{1th}_n(X,Y/U\wT{l})$ gives the required lift.
\end{proof}
The arguments in the proof above gives are enough itself for all equivalences of proposition \ref{prop:Freq} in the case of $U=\emptyset$, $l=0$.
To get the proof in the general case we need two extra definitions.

\begin{definition}[modified framed corr. $Fr^{\prime}$]\label{def:frcor}
For $Y\in \Sm_S$, and an open $U\subset Y$,
$Fr_n^\prime(Y/U\wT{l})$ is a pointed sheaf of sets with the sections
$Fr_n^\prime(X,Y/U\wT{l})$ for $X\in \Sch_S$
given by the equivalence classes of the data
$(Z,W,V, \phi,\psi, g)$,
where $V\to \A^n_X$ is an etale neighbourghood of a closed subscheme $W\subset \A^n_X$ over $X$, and 
$\alpha=(\phi,\psi,g)\colon V\to \A^n\times\A^l\times Y$ is a morphism of schemes such that 
$W = V\times_{\alpha,\A^{n+l},0} 0$, 
and $W\times_{g,Y,i} (Y\setminus U) = Z\amalg \hat Z$
$i\colon Y\setminus U\hookrightarrow Y$;  
all elements $(Z,V, \phi,\psi, g)$ with $Z=\emptyset$ are pointed;
the equivalence is up to a choice of the etale neighbourhood $V$.
\end{definition}
\begin{definition}[globally normally framed corr. $Fr^\text{g-nr}$]\label{def:normfrcor}
For $Y\in Sm_S$, and an open subscheme $U\subset Y$,
$Fr^\text{g-nr}_n(Y/U\wT{l})$ is a sheaf with the sections 
$Fr^\text{nr}(X,Y/U\wT{l})$ for $X\in \Sch_S$ given by the data
$(Z,W, \tau,\beta)$, where 
$Z\subset W\subset \A^n_X$ are closed,
$\tau\colon I(W)/(I^2(W))\simeq\cO^n(W)$,
$\beta\colon W\to \A^l\times Y$ such that $Z = W\times_{\A^l\times Y} (0\times (Y\setminus U))$;
the elements with $Z=\emptyset$ are pointed.
\end{definition}

\begin{lemma}\label{lm:Freq}
For any $Y\in \Sm_S$ and open $U\subset Y$  there are natural $\A^1$-Nis equivalences of motivic spaces 
$$
\begin{cases}
Fr^*_n(Y/U)\to Fr^\text{g-nr}_n(Y/U), *\in \{\text{pol},\text{Zar},\prime\}, n\in \mathbb Z_{\geq 0},\\
Fr^\text{g-nr}(Y/U)\to Fr^\text{nr}(Y/U).
\end{cases}$$
\end{lemma}
\begin{proof}
All arrows in the lemma are equivalences by the criteria given by proposition \ref{prop:LiftCriteria}.
Let us skip the verification of the closed glueing, which is straightforward. 
So if we check the lifting property on affines for the morphism $Fr^\text{g-nr}(Y/U)\to Fr^\text{nr}(Y/U)$, this would implies the equivalence.
\newcommand{\AffSm}{\mathrm{AffSm}}
Using lemma \ref{lm:Checkcov} we can reduce the question to the case of affine $Y$.

Consider an element of $Fr^\mathrm{g-nr}(X_0,Y/U\wT{l})\times_{Fr^\mathrm{nr}(X,Y/U\wT{l})}Fr^\mathrm{nr}(X,Y/U\wT{l})$. 
It is a set $(Z,W,\tau,g)$ such that $Z,W\subset \A^n_X$ are closed, 
\begin{multline*}
W\times (X-X_0)\subset Z(I^2(Z))\times_X (X-X_0), 
\tau \colon I(W)/(I(W_2)) \simeq \mathcal O(W)^n, \\
g\colon W\to \A^l\times Y, Z=W\times_{g,\A^l\times Y, 0\times i} 0\times (Y\setminus U),
\end{multline*}
where $W_2=W\times_X X_0 \cup Z(I^2(Z))$.
The trivialisation $\tau$ defines the regular map $\phi^\prime\colon  Z(I^2(W))\to \A^n$ such that $Z(I^2(W))\times_{\A^n} 0=W$.
By lemma \ref{lm:HensLiftSmooth}
there is an etale neighbourhood $V$ of $W$ in $\A^n_X$ and a map $(\tilde \phi,\tilde g)\colon V\to \A^l\times Y$ that is a lift of $\phi^\prime$ and $g$.
By lemma \ref{lm:EqforNeigh2} there are regular functions $f_i$ on $\A^n_X\times \A^m$, $i=1\dots m$, such that 
$Z(I^2(Z))\times 0\subset Z(f_1,\dots f_m)\subset \A^n_X\times \A^m$, $f_i\big|_{Z(I^2(Z\times 0)}=t_{n+i}$, where $t_{i}$ denotes coordinates on $\A^{n+m}_X$, and
$(Z(f_1,\dots f_m),0\times Z)\to (\A^n_X,Z)$ is a refinement of the neighbourhood $(V,Z)\to (\A^n_X,Z)$.
Let $W^\prime =Z(\phi^\prime,f_1,\dots f_m)\subset \A^{n+m}_X$, and $\tau^\prime\colon I(W^\prime)/I^2(W^\prime)\simeq \cO(W^\prime)^{n+m} $ is defined by the regular map $(\phi^\prime,f_1,\dots f_m)$.
Then $(0\times Z, W^\prime, \tau^\prime, (\psi^\prime,g^\prime)\big|_{W^\prime})\in Fr^\text{g-nr}_{n+m}(X,Y)$ is the required lift of $\sigma^m(Z,W,\tau,\psi,g)\in Fr^{nr}_{n+m}(X,Y/U\wT{l})$. 

The lifting property with respect to closed embeddings of affines for the arrows 
$Fr^*_n(Y/U)\to Fr^\text{g-nr}_n(Y/U)$
$*\in \{\text{Zar}, \prime\}$
follows immediately from the 
Chinese remainder theorem.

Finally, we need to check the lifting property for $Fr^{pol}_n(Y/U)\to Fr^\text{g-nr}_n(Y/U)$.
Let $X_0\to X$ be a closed embedding of affine schemes. 
Consider an element of $Fr^\mathrm{pol}(X_0,Y/U\wT{l})\times_{Fr^\mathrm{pol}(X,Y/U\wT{l})}Fr^\mathrm{g-nr}(X,Y/U\wT{l})$. 
It is given by a pair $c=(Z,W,\tau,\beta)\in Fr^\text{g-nr}_n(X,Y/U)$, $c_0=(\phi^0,\beta)\in Fr^\text{pol}_n(X_0,Y/U)$ such that $c_0$ is a lift of $c\big|_X$.
By Chinese remainder theorem there is a set of functions $\phi=(\phi_i)_{i=1\dots n}$, on $\A^n_X$ such that $\phi_i\big|_{\A^n\times X_0}=\phi^0_i$, and $\phi_i\big|_{Z(I^2(W))}$ are agreed with $\tau$.
Then $Z(\phi)=W\amalg \hat W$.
Let $\phi_{n+1}\in \cO(\A^{n+1}_X)$ be such that $\phi_{n+1}\big|_{Z(I^2(W))}=1$, $\phi_{n+1}\big|_{\hat W}=0$.
Now the correspondence
$(\phi,(t_{n+1}+1)\phi_{n+1}-1)\in Fr^\text{pol}_{n+1}(X,Y)$ is the required lift of $c$.
\end{proof}

\section{Appendix D: Lifting properties of smooth morphisms}
\begin{lemma}\label{lm:HensLiftSmooth}
Let $Y$ be a smooth 
over a base $S$.
Let $Z\hookrightarrow X$ be an affine henselian pair with $X$ being regular, and $g\colon Z\to Y$ be a morphism.
Then there is a lift of $g$ to a morphism $g^\prime\colon X\to Y$.
If in addition $Y$ is affine (and more generally quasi-projective) then $Y\to S$ satisfies the lifting property with respect to any affine henselian pair $Z\to X$. 
\end{lemma}
\begin{proof}
Firstly using Jouanolou's trick (Tomason's theorem) we reduce the question about the general $Y$ to the case of affine $Y$.
Namely, for an arbitrary $Y$ we consider the morphism $T=Y\times X\to X$ as a smooth scheme over an affine scheme $X$. Then by Jouanolou's trick \cite[Proposition 4.3, corollary 4.6]{Asok-JTt} there is an affine bundle $T^\prime\to T$ with affine $T^\prime$. Now Since $Z$ is affine $T^\prime\times_T Z\to Z$ is vector bundle, and there is a lift $Z\to T^\prime$ of the morphism $Z\to T$ defined by the morphism $Z\to Y$. Finally applying the lemma to the affine smooth $X$-scheme $T^\prime$ we get the lift $X\to T^\prime$ of the morphism $Z\to T^\prime$, and thus the composition $X\to T^\prime\to T\to Y$ gives the required lift of the morphism $Z\to Y$.

Assume $Y$ is affine and smooth over $S$. Let $\check{N}_Y\to Y$ be a vector bundle such that $\check{N}_Y\oplus N_Y$ is trivial, where $N_Y$ is relative normal bundle over $S$.
Then it is enough to prove the claim for $\check{N}_Y$ instead of $Y$.
So we can assume that normal (and tangent) bundle $N_Y$ of $Y$ is trivial, let $r=(r_1m\dots r_d)\colon \mathcal O(Y)\simeq N_Y$, $d=\dim Y$. 

By assumptions the morphism $Y\times_S X\to X$ is affine and admits a section $v\colon Z\to  Y\times_S X$ over $Z$.
We need to find a lift to a section $X\to  Y\times_S X$.
Consider regular functions $f_i$ on $Y\times_S X$ such that
$f_i\big|_{v(Z)}=0$, $f_i\big|_{Z_2}=r_i$, where $Z_2=Z(I^2(v(Z)))\times_X Z$.
Let $p\colon \overline V=Z(f_1,\dots f_d)\to X$ be the canonical projection.
Let $C\subset \overline V$ be the maximal closed subset such that fibres of ${\overline V}^2_X$ over $C$ are of dimension at least one. 
Let $V^\prime=\overline V-(C\cup (p^{-1}(Z)-v(Z)))$.
Then $V^\prime\to X$ is quasi-finite, and $V^\prime\times_X Z=v(Z)$.
It follows form the condition on $f_i$ that $V^\prime$ is smooth at $v(Z)$ over $S$.
Let $V\subset V^\prime$ be an open neighbourhood of $v(Z)$ that is smooth over $S$.
Since $V\times_X Z=v(Z)$ it follows that $e\colon V^{\prime\prime}\to X$ is unramified. 
Hence $e$ etale, and so there is a lift $X\to V$.
\end{proof}
Combining this with the notion of formally smoothness we get the following
\begin{corollary}
For a locally finite type $S$-scheme $Y$ the following are equivalent:
\item{1)} $Y$ is smooth;
\item{2)} $Y$ is formally smooth, i.e. it satisfies the lifting property with respect to a closed embeddings of affine schemes $Z\to X$ with $I^2(Z)=0$;
\item{3)} $Y$ satisfies the lifting property with respect to a henselian affine pairs $Z\to X$ with $X$ being regular.

If $Y$ is affine (or quasi-projective) in addition then the above conditions are equivalent to
\item{4)} $Y$ satisfies the lifting property with respect to a henselian affine pairs.
\end{corollary}

\end{document}